\documentclass[a4paper,10pt]{amsart}
\usepackage[utf8]{inputenc}
\usepackage[T1]{fontenc}
\usepackage[english]{babel}
\usepackage{textcomp}
\usepackage{lmodern}
\usepackage{amsthm}
\usepackage{amsmath}
\usepackage{amssymb}
\usepackage{mathrsfs}
\usepackage{enumerate}
\usepackage{dsfont}
\usepackage{color}
\usepackage{graphicx}
\usepackage{arydshln}
\usepackage{grffile}
\usepackage{float}
\usepackage[hidelinks]{hyperref}

\newcommand{\dd}{\,\mathrm{d}}
\newcommand{\RR}{\mathbb{R}}

\newcommand{\ds}{\displaystyle}
\newcommand{\dD}{\,\mathrm{d}}
\newcommand{\lla}{\left\langle}
\newcommand{\rra}{\right\rangle}
\newcommand{\eps}{\varepsilon}
\newcommand{\ffast}{f_\text{fast}}
\newcommand{\rhofast}{\rho_\text{fast}}
\newcommand{\jfast}{j_\text{fast}}
\newcommand{\trho}{{\tilde{\rho}}}
\newcommand{\trhofast}{{\tilde{\rho}}_\text{fast}}
\newcommand{\tjfast}{{\tilde{j}}_\text{fast}}

\newcommand\Max[1]{{#1}}

\newtheorem{theo}{Theorem}[section]
\newtheorem{prop}[theo]{Proposition}

\newtheorem{lem}[theo]{Lemma}
\newtheorem{rema}[theo]{Remark}
\newtheorem{conj}[theo]{Conjecture}

\title[Fokker-Planck modeling of gene expression]{A Fokker-Planck approach to the study of robustness in gene expression}
\date{\today}
\author{P. Degond}
\address[P. Degond]{Department of Mathematics, Imperial College London, South Kensington Campus, London, SW7 2AZ, UK}
\email{pdegond@imperial.ac.uk}
\author{M. Herda}
\address[M. Herda]{Inria, Univ. Lille, CNRS, UMR 8524 – Laboratoire Paul Painlevé, F-59000 Lille, France.}
\email{maxime.herda@inria.fr}
\author{S. Mirrahimi}
\address[S. Mirrahimi]{Institut de Math\'ematiques de Toulouse, UMR 5219, Universit\'e de Toulouse, CNRS, UPS IMT, F-31062 Toulouse Cedex 9, France.}
\email{sepideh.mirrahimi@math.univ-toulouse.fr}

\begin{document}
    
    \begin{abstract}
        We study several Fokker-Planck equations arising from a stochastic chemical kinetic  system modeling a gene regulatory network in biology. The densities solving the Fokker-Planck equations describe the joint distribution of the mRNA and $\mu$RNA content in a cell. We provide theoretical and numerical \Max{evidence}   that the robustness of the gene expression is increased in the presence of $\mu$RNA. At the mathematical level, increased robustness shows in a smaller coefficient of variation of the marginal density of the mRNA in the presence of $\mu$RNA. These results follow from explicit formulas for solutions. Moreover, thanks to dimensional analyses  and numerical simulations we provide qualitative insight into the role of each parameter in the model. As the increase of gene expression level comes from the underlying stochasticity in the models, we eventually discuss the choice of noise in our models and its influence on our results. 
        
        \smallskip
        
        
        \smallskip
        
        
    \end{abstract}

    \maketitle
    \tableofcontents
    
    \section{Introduction}

    This paper is concerned with a mathematical model for a gene regulatory network involved in the regulation of DNA transcription. DNA transcription is part of the mechanism by which a sequence of the nuclear DNA is translated into the corresponding protein. The transcription is initiated by the binding of a transcription factor, which is usually another protein, onto the gene's DNA-binding domain. Once bound, the transcription factor promotes the transcription of the nuclear DNA into a messenger RNA (further denoted by mRNA), which, once released, is \Max{translated} 
    into the corresponding protein by the ribosomes. This process is subject to a high level of noise due to the large variability of the conditions that prevail in the cell and the nucleus at the moment of the transcription. Yet, a rather stable amount of the final protein is needed for the good operation of the cell. The processes that regulate noise levels and maintain cell homeostasis have been scrutinized for a long time. Recently, micro RNAs (further referred to as $\mu$RNAs) have occupied the front of the scene. These are very short RNAs which do not code for proteins. Many different sorts of $\mu$RNAs are involved in various epigenetic processes. But one of their roles seems precisely the reduction of noise level in DNA transcription. In this scenario, the $\mu$RNAs are synthesized together with the mRNAs. Then, some of the synthesized $\mu$RNAs bind to the mRNAs and de-activate them. These $\mu$RNA-bound mRNA become unavailable for protein synthesis. It has been proposed that this paradoxical mechanism which seems to reduce the efficiency of DNA transcription may indeed have a role in noise regulation (see \cite{bleris2011synthetic,ebert2012roles,blevins2015micrornas} and the review \cite{herranz2010micrornas}). The goal of the present contribution is to propose a mathematical model of the $\mu$RNA-mRNA interaction and to use it to investigate the role of $\mu$RNAs as potential noise regulators. 
    
    Specifically, in this paper, we propose a stochastic chemical kinetic model for the mRNA and $\mu$RNA content in a cell. The production of mRNAs by the transcription factor and their inactivation through $\mu$RNA binding are taken into account. More precisely, our model is a simplified version of the circuit used in \cite[Fig. 2A and 2A']{osella_2011_role}. We consider a ligand involved in the production of both an mRNA and a $\mu$RNA, the $\mu$RNA having the possibility to bind to the mRNA and deactivate it. By contrast to \cite{osella_2011_role}, we disregard the way the ligand is produced and consider that the ligand is such that there is a constant production rate of both mRNA and $\mu$RNA. A second difference to \cite{osella_2011_role} is that we disregard the transcription step of the mRNA into proteins. While \cite{osella_2011_role} proposes to model the $\mu$RNA as acting on \Max{translation}, 
     we assume that the $\mu$RNA directly influences the number of mRNA available for transcription. Therefore, we directly relate the gene expression level to the number of $\mu$RNA-free mRNA also referred to as the number of unbound mRNA. In order to model the stochastic variability in the production of the RNAs, a multiplicative noise is added to the production rate at all time. From the resulting system of stochastic differential equations, we introduce the joint probability density for mRNA and µRNA which solves a deterministic Fokker-Planck equation. The mathematical object of interest is the stationary density solving the Fokker-Planck equation and more precisely the marginal density of the mRNA. The coefficient of variation (also called cell-to-cell variation) of this mRNA density, which is its standard deviation divided by its the expectation, is often considered as the relevant \Max{criterion} 
     for measuring the robustness of gene expression (see for instance \cite{osella_2011_role}). 
    
    Our main goal in this contribution is to provide theoretical and numerical \Max{evidence}   that the robustness of the gene expression is increased in the presence of $\mu$RNA. At the theoretical level we derive a number of analytical formulas either for particular subsets of parameters of the model or under some time-scale separation hypotheses. From these formulas we can easily compute the cell to cell variation numerically and verify the increased robustness of gene expression when binding  with \Max{$\mu$RNA}  happens in the model. For general sets of parameters, the solution cannot be computed analytically. However we can prove well-posedness of the model and solve the PDE with a specifically designed numerical scheme. From the approximate solution, we compute the coefficient of variation and verify the hypothesis of increased gene expression.

    Another classical approach to the study of noise in gene regulatory networks is through the chemical master equation \cite{kampen_1981_stochastic} which is solved numerically by means of Gillespie’s algorithm \cite{gillespie_1976_general}, see e.g.
\cite{bosia2012gene, osella_2011_role}. Here, we use a stochastic chemical kinetic model through its associated Kolmogorov-Fokker-Planck equation. Chemical kinetics is a good approximation of the chemical master equation when the number of copies of each molecule is large. This is not the case in a cell where sometimes as few as a 100 copies of some molecules are available. Specifically, including a stochastic term in the chemical kinetic approach is a way to retain some of the randomness of the process while keeping the model complexity tractable. This ultimately leads to a Fokker-Planck model for the joint distribution of mRNAs and $\mu$RNAs. In \cite{degond2019uncertainty}, a similar chemical kinetic model is introduced with a different modelling of stochasticity. The effect of the noise is taken into account by adding some uncertainty in the (steady) source term and the initial data. The authors are interested in looking at how this uncertainty propagates to the mRNA content and in comparing this uncertainty between situations including µRNA production or not. The uncertainty is modeled by random variables with given probability density functions. Compared to \cite{degond2019uncertainty}, the Fokker-Planck approach has the advantage that the random perturbations do not only affect the initial condition and the source term, but are present at all times and \Max{vary} 
through time. We believe that this is coherent with how stochasticity in a cell arises through time-varying ecological or biological factors.

While Fokker-Planck equations are widely used models in mathematical biology \cite{perthame_2015_parabolic}, their use for the study of gene regulatory network is, up to our knowledge, scarce (see e.g. \cite{lotstedt2006dimensional}). Compared to other approaches, the Fokker-Planck model enables us  to derive analytical formulas for solutions in certain cases. This is particularly handy for understanding the role of each parameter in the model, calibrating them from real-world data and perform fast numerical computations. Nevertheless, in the general case, the theoretical study and the numerical simulation of the model remains challenging because of the unboundedness of the drift and diffusion coefficients. We believe that we give below all the tools for handling these difficulties, and that our simple model \Max{provides} 
a convincing mathematical interpretation of the increase of gene expression in the presence of  $\mu$RNAs.
    
    The paper is organized as follows. In Section~\ref{sec:models}, we introduce the system of SDEs and the corresponding Fokker-Planck models. In Section~\ref{sec:solve}, we discuss the well-posedness of the Fokker-Planck equations and derive analytical formulas for solutions under some simplifying hypotheses. In Section~\ref{sec:noise_fast}, we use the analytical formulas for solutions to give mathematical and numerical proofs of the decrease of cell-to-cell variation in the presence of $\mu$RNA. In Section~\ref{sec:noise_main}, we propose a numerical scheme for solving the main Fokker-Planck model and gather further \Max{evidence}   confirming the hypothesis of increased gene expression from the simulations. Finally, in Section~\ref{sec:comments} we discuss the particular choice of multiplicative noise (\textit{i.e.} the diffusion coefficient in the Fokker-Planck equation) in our model. In the appendix, we derive weighted Poincaré inequalities for gamma and inverse-gamma distributions which are useful in the analysis of Section~\ref{sec:solve}. \Max{The code used for numerical simulations in this paper is publicly available on GitLab \cite{code}.}

    \section{Presentation of the models}\label{sec:models}
    
    In this section, we introduce three steady Fokker-Planck models whose solutions describe the distribution of unbound mRNA and $\mu$RNA within a cell. The solutions to these equations can be interpreted as the probability density functions associated with the steady states of  stochastic chemical kinetic systems describing the production and destruction of mRNA and $\mu$RNA. In Section~\ref{sec:binding} we introduce the main model for which the consumption of RNAs is either due to external factors in the cell (\Max{translation}, 
    \emph{etc.})  or to binding between the two types of mRNA and $\mu$RNA. Then, for comparison, in Section~\ref{sec:free} we introduce the same model without binding between RNAs. Finally in Section~\ref{sec:fast}, we derive an approximate version of the first  model, by considering that reactions involving $\mu$RNAs are infinitely faster than those involving mRNAs, which amplifies the binding phenomenon and mathematically allows for the derivation of analytical formulas for solutions. The latter will be made explicit in Section~\ref{sec:solve}.
    
    \subsection{Dynamics of mRNA and $\mu$RNA with binding}\label{sec:binding}
    We denote by $r_t$ the number of unbound mRNA and $\mu_t$ the number of unbound \Max{$\mu$RNA}  of a given cell at time t. The kinetics of unbound mRNA and $\mu$RNA is then given by
    the following stochastic differential equations
    \begin{equation}
        \left\{
        \begin{array}{ccccc}
            \dd r_t&=&(c_r - c\, r_t\, \mu_t - k_r\, r_t)\, \dd t &+&\ \sqrt{2\sigma_r}\, r_t\, \dd B_t^1\,,\\[1em]
            \dd\mu_t\, &=&\, (c_\mu\, - c\, r_t\, \mu_t - k_\mu\, \mu_t)\, \dd t &+& \sqrt{2\sigma_\mu}\, \mu_t\, \dd B_t^2\,,
        \end{array}
        \right.
        \label{eqrmu}
    \end{equation}
    with $c_r$, $c_\mu$, $k_r$, $k_\mu$, $\sigma_r$, $\sigma_\mu$ being some given positive constants and $c$ being a given non-negative constant.  Let us detail the meaning of each term in the modeling. The first term of each equation models the constant production of mRNA (\emph{resp.} $\mu$RNA) by the ligand at a rate $c_r$  (\emph{resp.} $c_\mu$). The second term models the binding of the $\mu$RNA to the mRNA. Unbound mRNA and $\mu$RNA are consumed by this process at the same rate. The rate increases with both the number of mRNA and $\mu$RNA. In the third term, the parameters $k_r$ and $k_\mu$ are the rates of consumption of the unbound mRNA or $\mu$RNA by various decay mechanisms. The last term in both equations represents stochastic fluctuations in the production and destruction mechanisms of each species. It relies on a white noise $\dd B_t / \dd t$ where $B_t = (B_t^1,B_t^2)$ is a \Max{two-dimensional} 
    standard Brownian motion. The intensity of the stochastic noise is quantified by the parameters $\sqrt{2\sigma_r}\,r_t$ and $\sqrt{2\sigma_\mu}\,\mu_t$. Such a choice of multiplicative noise ensures that $r_t$ and $\mu_t$ remain non-negative along the dynamics. { The Brownian motions $B_t^1$ and $B_t^2$ are uncorrelated. The study of correlated noises or the introduction of extrinsic noise sources would be interesting, but will be discarded here.} 
    
    In this paper we are interested in the invariant measure of \eqref{eqrmu} rather than the time dynamics described by the above SDEs. From the modelling point of view, we are considering a large number of identical cells and we assume that mRNA and $\mu$RNA numbers evolve according to \eqref{eqrmu}. Then we measure the distribution of both RNAs among the population, when it has reached a steady state $f \equiv f(r,\mu)$. According to Itô's formula, the steady state should satisfy the following steady Fokker-Planck equation
    \begin{equation}
        \left\{
        \begin{array}{l}
            L f(r,\mu) \ =\ 0\,,\quad  (r,\mu)\in\Omega = (0,\infty)^2\,,\\[.75em]
            \ds\int_{\Omega}f(r,\mu)\dd r\dd \mu \ =\ 1\,,\quad f(r,\mu)\ \geq\ 0\,.
        \end{array}
        \right.
        \label{eq:FPstat}
    \end{equation}
    where the Fokker-Planck operator is given by 
    \begin{multline}
        Lf(r,\mu)\ :=\ \partial_r\left[\partial_r (\sigma_rr^2f) -(c_r - c\,r\,\mu - k_r\,r)f\right]\\[.25em]+ \partial_\mu\left[ \partial_\mu( \sigma_\mu\mu^2f) -(c_\mu - c\,r\,\mu - k_\mu\,\mu)f\right]\,.
        \label{eq:defFP}
    \end{multline}
    Since we do not model the protein production stage, we assume that the observed distribution of gene expression level is proportional to the marginal distribution of mRNA, \emph{i.e.}
    \[
    \rho(r) = \int_0^\infty f(r,\mu) \dd \mu\,.
    \]
    By integration of \eqref{eq:FPstat} in the $\mu$ variable, $\rho$ satisfies the equation 
    \begin{equation}
        \partial_r\left[\partial_r(\sigma_r\ r^2\ \rho) - (\ c_r - c\ r\ j_\mu(r) - k_r\ r)\rho  \right] = 0.
        \label{eq:rho}
    \end{equation}
    The quantity $j_\mu(r)$ is the conditional expectation of the number of $\mu$RNA within the population in the presence of $r$ molecules of mRNA and it is given by
    \begin{equation}\label{eq:jmu}
        j_\mu(r) = \frac{1}{\rho(r)}\int \mu\ f(r,\mu) \dd \mu\,.
    \end{equation}
{ Before ending this paragraph, we note an alternate way to derive the Fokker-Planck equation \eqref{eq:FPstat} from the chemical master equation through the chemical Langevin equation. We refer the interested reader to \cite{gillespie2000chemical}. }
		
    \subsection{Dynamics of free mRNA without binding}\label{sec:free}
    In the case where there is no $\mu$RNA binding, namely when $c=0$, the variables $r$ and $\mu$ are independent. Thus, the densitites of the invariant measures satisfying \eqref{eq:FPstat} are of the form 
    \[f_0(r,\mu)=\rho_0(r)\lambda(\mu)\,,\]
    where $\lambda(\mu)$ is the density of the marginal distrubution of $\mu$RNA. From the modelling point of view, it corresponds to the case where there is no feed-forward loop from $\mu$RNA. Therefore, only the dynamics on mRNA, and thus $\rho_0$, is of interest in our study. It satisfies the following steady Fokker-Planck equation obtained directly from \eqref{eq:rho}, 
    \begin{equation}
        \left\{
        \begin{array}{l}
            \partial_r\left[\sigma_r\partial_r (r^2\rho_0) -(c_r - k_r\,r)\rho_0\right]\ =\ 0\,,\\[.5em]
            \ds\int_0^\infty\rho_0(r)\dd r\ =\ 1\,,\quad \rho_0(r)\geq0\,.
        \end{array}
        \right.
        \label{eq:FP2stat}
    \end{equation}
    It can be solved explicitly as we will discuss in Section~\ref{s:exist_without}.

    \subsection{Dynamics with binding and fast $\mu$RNA}\label{sec:fast}
    The Fokker-Planck equation \eqref{eq:FPstat} cannot be solved explicitly. However, one can make some additional assumptions in order to get an explicit invariant measure providing some insight into the influence of the binding mechanism with $\mu$RNA. This is the purpose of the model considered hereafter.
    
    Let us assume the $\mu$RNA-mRNA binding rate, the $\mu$RNA decay and the noise on $\mu$RNA are large. Since the sink term of the $\mu$RNA equation is large, it is also natural to assume that the $\mu$RNA content is small. Mathematically, we assume the following scaling
    \[
     c = \frac{\tilde{c}}{\varepsilon}\,,\quad k_\mu = \frac{\tilde{k}_\mu}{\varepsilon}\,,\quad \sigma_\mu = \frac{\tilde{\sigma}_\mu}{\varepsilon}\,,\quad \mu_t = \varepsilon\tilde{\mu_t}\,,
    \]
    for some small constant $\varepsilon>0$. Then $(r_t,\tilde{\mu}_t)$ satisfies
    \[
        \left\{
        \begin{array}{ccccc}
            \dd r_t&=&(c_r - \tilde{c}\, r_t\, \tilde{\mu}_t - \tilde{k}_r\, r_t)\, \dd t &+&\ \sqrt{2\sigma_r}\, r_t\, \dd B_t^1\,,\\[1em]
            \varepsilon\,\dd\,\tilde{\mu}_t\, &=&\, (c_\mu\, - \tilde{c}\, r_t\, \tilde{\mu}_t - \tilde{k}_\mu\, \tilde{\mu}_t)\, \dd t &+& \sqrt{2\,\varepsilon\,\tilde{\sigma}_\mu}\, \tilde{\mu}_t\, \dd B_t^2\,,
        \end{array}
        \right.
    \]
    whose corresponding steady Fokker-Planck equation for the invariant measure then writes, dropping the tilde,
    \begin{multline*}
        \partial_r\left[\partial_r (\sigma_rr^2f_\eps) -(c_r - c\,r\,\mu - k_r\,r)f_\eps\right]\\[.25em]+\frac{1}{\eps} \partial_\mu\left[ \partial_\mu( \sigma_\mu\mu^2f_\eps) -(c_\mu - c\,r\,\mu - k_\mu\,\mu)f_\eps\right]\ =\ 0
    \end{multline*}
    In the limit case where $\eps\to 0$, one may expect that at least formally, the density $f_\eps$ converges to a limit density $\ffast$ satisfying
    \[
    \partial_\mu\left[ \partial_\mu( \sigma_\mu\mu^2\ffast) -(c_\mu - c\,r\,\mu - k_\mu\,\mu)\ffast\right]\ =\ 0\,.
    \]
    As $r$ is only a parameter in the previous equation and since the first marginal of $f_\eps$ still satisfies \eqref{eq:rho} for all $\eps$, one should have (formally)
    \begin{equation}\label{eq:fast}
        \left\{
        \begin{array}{l}
            \ffast(r,\mu)\ =\ \rhofast(r)M(r,\mu)\ \geq\ 0\,,\\[.75em]
            \partial_\mu\left[ \partial_\mu( \sigma_\mu\mu^2M) -(c_\mu - c\,r\,\mu - k_\mu\,\mu)M\right]\ =\ 0\,,\\[.75em]
            \partial_r\left[\partial_r(\sigma_r\ r^2\ \rhofast) - (\ c_r - c\ r\ \jfast(r) - k_r\ r)\rhofast  \right] = 0\,,\\[.75em]
            \ds\int_{0}^\infty\rhofast(r)\dd r\ =\ 1\,,\ \int_{0}^\infty M(r,\mu)\dd \mu\ =\ 1\,,\\[.75em] 
            \ds\jfast(r)\ =\ \int_0^\infty \mu\ M(r,\mu) \dd \mu\,.    
        \end{array}
        \right.
    \end{equation}

    \section{Well-posedness of the models and analytical formulas for solutions}\label{sec:solve}
    
    In this section, we show that the three previous models are well-posed. For the Fokker-Planck equations \eqref{eq:FP2stat} and \eqref{eq:fast}, we \Max{explicitly} 
    compute the solutions. They involve inverse gamma distributions.
    
    \subsection{Gamma and inverse gamma distributions}\label{sec:gamma}
    The expressions of the gamma and inverse gamma probability densities are respectively
    \begin{equation}
        \gamma_{\alpha,\beta}(x)\ =\ C_{\alpha,\beta}\,\frac{}{}x^{\alpha-1}\exp\left(-\beta x\right)\,,
        \label{eq:gamma}
    \end{equation}
    and
    \begin{equation}
        \qquad g_{\alpha,\beta}(y)\ =\ \,\frac{C_{\alpha,\beta}}{y^{1+\alpha}}\exp\left(-\frac{\beta}{y}\right)\,,
        \label{eq:invgamma}
    \end{equation}
    for $x,y\in(0,\infty)$. The normalization constant is given by $C_{\alpha,\beta} =  \beta^{\alpha}/\Gamma(\alpha)$ where $\Gamma$ is the Gamma function. Observe that by the change of variable $y = 1/x$ one has
    \[
    g_{\alpha,\beta}(y)\dD y\ =\ \gamma_{\alpha,\beta}(x)\dD x
    \]
    which justifies the terminology. Let us also recall that the first and second moments of the inverse gamma distribution are
    \begin{equation}\label{eq:firstmoment}
        \int_0^\infty y\,g_{\alpha,\beta}(y)\dd y\ =\ \frac{\beta}{\alpha-1}\,,\quad\text{ if } \alpha>1\,,\ \beta>0\,, 
    \end{equation}
    \begin{equation}\label{eq:secondmoment}
        \int_0^\infty y^2\,g_{\alpha,\beta}(y)\dd y\ =\ \frac{\beta^2}{(\alpha-1)(\alpha-2)}\,,\quad\text{ if } \alpha>2\,,\ \beta>0\,, 
    \end{equation}
    Interestingly enough, we can show (see Appendix~\ref{s:weightPoinc} for details and additional results) that inverse gamma distributions with finite first moment ($\alpha>1$) satisfy a (weighted) Poincaré inequality. The proof of the following proposition is done in Appendix~\ref{s:weightPoinc} among more general considerations.
    \begin{prop}
        Let $\alpha>1$ and $\beta>0$. Then, for any function $v$ such that the integrals make sense, one has 
        \begin{equation}
            \int_0^\infty|v(y) - \lla\,v\,\rra_{g_{\alpha,\beta}}|^2\,g_{\alpha,\beta}(y)\,\dD y\ \leq\ \frac{1}{\alpha-1}\,\int_0^\infty|v'(y)|^2\,g_{\alpha,\beta}(y)\,y^2\,\dD y\,,
            \label{eq:weightPoincare_IG_2}
        \end{equation}
        where for any probability density $\nu$ and any function $u$ on $(0,\infty)$, the notation $\lla\,u\,\rra_{\nu}$ denotes $\int u\nu$.
    \end{prop}
    
    \subsection{Explicit mRNA distribution without binding}\label{s:exist_without}
    In the case of free mRNAs, a solution to \eqref{eq:FP2stat} can be computed explicitely and takes the form of an inverse gamma distribution.
    
    \begin{lem}\label{p:classicalrho}
        The following inverse gamma distribution
        \begin{equation}\label{eq:rho0}
        \rho_0(r)\ =\ g_{1+\frac{k_r}{\sigma_r}, \frac{c_r}{\sigma_r}}(r)\ =\ \,C_{1+\frac{k_r}{\sigma_r}, \frac{c_r}{\sigma_r}}\,\frac{1}{r^{2+\frac{k_r}{\sigma_r}}}\exp\left(-\frac{c_r}{\sigma_r\,r}\right)
        \end{equation}
        is the only classical solution to \eqref{eq:FP2stat}.
    \end{lem}
    \begin{proof}
        First observe that
        \[
        \partial_r\left[\sigma_r\partial_r (r^2\rho_0) -(c_r - k_r\,r)\rho_0\right]\ =\ \partial_r\left[\sigma_r\,r^2\,g_{1+\frac{k_r}{\sigma_r}, \frac{c_r}{\sigma_r}}\partial_r \left(\rho_0\,g_{1+\frac{k_r}{\sigma_r}, \frac{c_r}{\sigma_r}}^{-1}\right)\right]\,.
        \]
        Therefore a solution of \eqref{eq:gamma} must be of the form
        \[
        \rho_0(r)\ =\ C_1\,g_{1+\frac{k_r}{\sigma_r}, \frac{c_r}{\sigma_r}}\int_{1}^r\sigma_r^{-1}r^{-2}g_{1+\frac{k_r}{\sigma_r}, \frac{c_r}{\sigma_r}}^{-1}(r)\dd r + C_2\,g_{1+\frac{k_r}{\sigma_r}, \frac{c_r}{\sigma_r}}\,,
        \]
        for some constants $C_1,C_2$. The first term decays like $1/r$ at infinity, thus the only probability density $\rho_0$ of this form is obtained for $C_1=0$ and $C_2=1$.
        
    \end{proof}
    The Poincaré inequality \eqref{eq:weightPoincare_IG_2} tells us that the solution of \Max{Lemma~\ref{p:classicalrho}} 
    is also the only (variational) solution in the appropriate weighted Sobolev space. Indeed, we may introduce the natural Hilbert space associated with Equation~\eqref{eq:FP2stat}, 
    \[
    X_{\alpha,\beta}\ =\ \{v:(0,\infty)\to\RR\,,\ \|v\|_{X_{\alpha,\beta}} <\infty\}
    \]
    with a squared norm given by
    \[
    \|v\|_{X_{\alpha,\beta}}^2\ =\ \int_0^\infty \left(\left[\left(v/g_{\alpha,\beta}\right)(y)\right]^2 + \left[(v/g_{\alpha,\beta})'(y)\right]^2\,y^2\right)\,g_{\alpha,\beta}(y)\dD y\,.
    \]
    Then the following uniqueness result holds.
    \begin{lem}\label{lem:classic_rho}
        The classical solution $\rho_0 = g_{1+\frac{k_r}{\sigma_r}, \frac{c_r}{\sigma_r}}$ is the only solution of \eqref{eq:FP2stat} in $X_{1+\frac{k_r}{\sigma_r}, \frac{c_r}{\sigma_r}}$.
    \end{lem}
    \begin{proof}
        If $\rho_0$ and $\tilde{\rho_0}$ are two solutions of \eqref{eq:FP2stat}, a straightforward consequence of \eqref{eq:weightPoincare_IG_2} is that $\|\rho_0-\tilde{\rho_0}\|_{X_{1+\frac{k_r}{\sigma_r}, \frac{c_r}{\sigma_r}}}=0$. This is obtained by integrating the difference between the equation on $\rho_0$ and $\tilde{\rho_0}$  against $(\rho_0-\tilde{\rho_0})g_{1+\frac{k_r}{\sigma_r}, \frac{c_r}{\sigma_r}}^{-1}$.
    \end{proof}

    Another consequence of the Poincaré inequality is that if we consider the time evolution associated with the equation \eqref{eq:FP2stat} then solutions converge exponentially fast towards the steady state $\rho_0$. This justifies our focus on the stationary equations. The transient regime is very short and equilibrium is reached quickly. We can quantify the rate of convergence in terms of the parameters.
    \begin{prop}
        Let $\xi$ solve the Fokker-Planck equation
        \[
         \partial_t\xi(t,r)\ =\ \partial_r\left[\sigma_r\partial_r (r^2\xi(t,r)) -(c_r - k_r\,r)\xi(t,r)\right]\,,
        \]
        starting from the probability density $\xi(0,r,\mu) = \xi^\text{in}(r,\mu)$. Then for all $t\geq0$,
        \begin{multline*}
        \int_0^\infty(\xi(t,r)-\rho_0(r))^2\,g_{1+\frac{k_r}{\sigma_r}, \frac{c_r}{\sigma_r}}^{-1}(r)\dd r\\
        \leq\ e^{-\frac{k_r}{\sigma_r}\,t}\,\int_0^\infty(\xi^\text{in}(r)-\rho_0(r))^2\,g_{1+\frac{k_r}{\sigma_r}, \frac{c_r}{\sigma_r}}^{-1}(r)\dd r\,.
        \end{multline*}
    \end{prop}
    \begin{proof}
       Observe that $\xi-\rho_0$ solves the unsteady Fokker-Planck equation, so that by multiplying the equation by $(\xi-\rho_0)\,g_{1+\frac{k_r}{\sigma_r}, \frac{c_r}{\sigma_r}}^{-1}$ and integrating in $r$ one gets
       \begin{multline*}
       \frac{\dd}{\dd t} \int_0^\infty(\xi(t,r)-\rho_0(r))^2\,g_{1+\frac{k_r}{\sigma_r}, \frac{c_r}{\sigma_r}}^{-1}(r)\dd r \\ \,+\, \int_0^\infty\left|\partial_r\left(\frac{\xi(t,\cdot)-\rho_0}{g_{1+\frac{k_r}{\sigma_r}, \frac{c_r}{\sigma_r}}}\right)(r)\right|^2\,g_{1+\frac{k_r}{\sigma_r}, \frac{c_r}{\sigma_r}}(r)\,r^2\,\dd r\ =\ 0\,.
       \end{multline*}
       Then by using the Poincaré inequality \eqref{eq:weightPoincare_IG_2} and a Gronwall type argument, one gets the result. 
    \end{proof}

    \subsection{Explicit mRNA distribution in the presence of fast $\mu$RNA}
    
    Now we focus on the \Max{solution} 
    of \eqref{eq:fast}. The same arguments \Max{as} 
    those establishing Lemma~\ref{lem:classic_rho} show that the only function $M$ satisfying \eqref{eq:fast} is the following inverse gamma distribution
    \begin{equation}\label{eq:M}
        M(r,\mu)\ =\ g_{1+\frac{k_\mu}{\sigma_\mu}+\frac{c}{\sigma_\mu}r, \frac{c_\mu}{\sigma_\mu}}(\mu)\,.
    \end{equation}
    Then an application of \eqref{eq:firstmoment} yields
    \begin{equation}\label{eq:jfast}
        \jfast(r) = \frac{c_\mu}{k_\mu+cr}\,.
    \end{equation}
    It remains to find $\rhofast$ which is a probability density solving the Fokker-Planck equation 
    \[
    \partial_r\left[\partial_r(\sigma_r\ r^2\ \rhofast) - (\ c_r - \frac{c_\mu\,c\,r}{k_\mu+cr} - k_r\ r)\rhofast  \right] = 0\,.
    \]
    Arguing as in the proof of Lemma~\ref{p:classicalrho}, one observe that integrability properties force $\rhofast$  to actually solve
    \[
    \partial_r(\sigma_r\ r^2\ \rhofast) - (\ c_r - \frac{c_\mu\,c\,r}{k_\mu+cr} - k_r\ r)\rhofast = 0\,.
    \]
    which yields 
    \begin{equation}\label{eq:rhofast}
        \rhofast(r)\ =\ C\,\left(1+\frac{k_\mu}{cr}\right)^{\frac{c\,c_\mu}{\sigma_r\,k_\mu}}\,\frac{1}{r^{2+\frac{k_r}{\sigma_r}}}\,\exp\left(-\frac{c_r}{\sigma_r r}\right)
    \end{equation}
    where $C\equiv C(c_r, c_\mu, k_r, k_\mu, \sigma_r, \sigma_\mu,c)$ is a normalizing constant making $\rhofast$ a probability density function.
    
    \subsection{Well-posedness of the main Fokker-Planck model}
    Now we are interested in the well-posedness of \eqref{eq:FPstat}, for which we cannot derive explicit formulas anymore. Despite the convenient functional framework introduced in Section~\ref{s:exist_without}, classical arguments from elliptic partial differential equation theory do not seem to be adaptable to the case $c>0$. The main obstruction comes from an incompatibility between the natural decay of functions in the space $X_{\alpha,\beta}$ and the rapid growth of the term $c\,r\,\mu$ when $|(r,\mu)|\to \infty$. 
    
    However, thanks to the results of \cite{huang_2015_steady} focused specifically on Fokker-Planck equations, we are able to prove well-posedness of the steady Fokker-Planck equation \eqref{eq:FPstat}. The method is based on finding a Lyapunov function for the adjoint of the Fokker-Planck operator and relies on an integral identity proved by the same authors in \cite{huang_2015_integral}. 
    \Max{The interested reader may also find additional material and a comprehensive exposition concerning the analysis of general Fokker-Planck equations for measures in \cite{bogachev_2015_Fokker}.}

    First of all let us specify the notion of solution. A weak solution to \eqref{eq:FPstat} is an integrable function $f$ such that
    \begin{equation}\label{eq:weakform}
        \left\{
        \begin{array}{ll}
            \ds\int_{\Omega}f(r,\mu)\mathcal{L}\varphi(r,\mu)\ =\ 0\,,&\text{for all}\ \varphi\in\mathcal{C}^\infty_c(\Omega)\,,\\[.75em]
            \ds\int_{\Omega}f(r,\mu) \ =\ 1\,,& f(r,\mu)\ \geq\ 0\,,
        \end{array}
        \right.
    \end{equation}
    where the adjoint operator is given by
    \begin{multline}
        \mathcal{L}\varphi(r,\mu)\ :=\ \sigma_rr^2\partial_{rr}^2\varphi +(c_r - c\,r\,\mu - k_r\,r)\partial_r\varphi\\ + \sigma_\mu\mu^2\partial_{\mu\mu}^2 \varphi +(c_\mu - c\,r\,\mu - k_\mu\,\mu)\partial_\mu \varphi\,.
    \end{multline}
    
    A reformulation and combination of \cite[Theorem A and Proposition 2.1]{huang_2015_steady} provides the following result.
    \begin{prop}[\cite{huang_2015_steady}]\label{p:existLyap}
        Assume that there is a smooth function $U:\Omega\rightarrow[0,+\infty)$, called \emph{Lyapunov function} with respect to $\mathcal{L}$, such that 
        \begin{equation}\label {eq:Lyap1}
            \lim_{(r,\mu)\to\overline{\partial\Omega}}U(r,\mu)\ =\ +\infty\,,
        \end{equation}
        and
        \begin{equation}\label {eq:Lyap2}
            \lim_{(r,\mu)\to\overline{\partial\Omega}}\mathcal{L}U(r,\mu)\ =\ -\infty\,,
        \end{equation}
        where $\overline{\partial\Omega} = \partial\Omega\cup(\{+\infty\}\times\RR_+)\cup(\RR_+\times\{+\infty\})$. Then there is a unique $f$ satisfying \eqref{eq:weakform}. Moreover $f\in W^{1,\infty}_\mathrm{loc}(\Omega)$.
    \end{prop}
\Max{\begin{rema}
     The method of Lyapunov functions is a standard tool for proving well-posedness of many problems in the theory of ordinary differential equations, dynamical systems... For diffusion processes and Fokker-Planck equations its use dates back to Has'minskii \cite{hasminski_1960_ergodic}. We refer to \cite[Chapter 2]{bogachev_2015_Fokker}, \cite{huang_2015_steady} and references therein for further comments on the topic. Let us stress however that  Lyapunov functions are not related (at least directly) to the Lyapunov (or entropy) method for evolution PDEs in which one shows the monotony of a functional to quantify long-time behavior.
     \end{rema}
}
    
    \Max{\begin{rema}
          Thanks of the degeneracy of the diffusivities at $r=0$ and $\mu=0$ and the Lyapunov function condition, one doesn't need supplementary boundary conditions in \eqref{eq:weakform} for the problem to have a unique solution. This is different from standard elliptic theory where boundary conditions are necessary to define a unique solution when the domain and the coefficients are bounded with uniformly elliptic diffusivities. Further comments may be found in \cite{huang_2015_steady}.
         \end{rema}}
    
    \begin{lem}\label{l:lyapFP}
        Choose any two constants $b_r > \frac{c}{k_\mu}$ and $b_\mu > \frac{c}{k_r}$. Then, the function $U:\Omega\rightarrow\RR$ defined by 
        \[
        U(r,\mu)\ =\ b_rr-\ln(b_rr)+b_\mu\mu-\ln(b_\mu\mu)
        \]
        is a Lyapunov function with respect to $\mathcal{L}$ (\emph{i.e.} it is positive on $\Omega$ and it satisfies \eqref{eq:Lyap1} and \eqref{eq:Lyap2}).
    \end{lem}
    \begin{proof}
        First observe that condition \eqref{eq:Lyap1} is clearly satisfied. Also, $U$ is minimal at $(b_r^{-1}, b_\mu^{-1})$ where it takes the value $2$ and thus it is positive on $\Omega$.  Finally a direct computation yields
        \begin{multline*}
            \mathcal{L}U(r,\mu)\ =\ (\sigma_r+\sigma_\mu+b_rc_r+b_\mu c_\mu+k_r+k_\mu) \\-\frac{c_r}{r}-\frac{c_\mu}{\mu} -(b_rk_r-c)r-(b_\mu k_\mu-c)\mu-c r\mu(b_r+b_\mu)\,,
        \end{multline*}
        and \eqref{eq:Lyap2} follows.
    \end{proof}
    \Max{Now we state our well-posedness result for the Fokker-Planck equation \eqref{eq:weakform}.}
    \begin{prop}
        There is a unique weak solution $f$ to the steady Fokker-Planck equation \eqref{eq:FPstat}. Moreover, \Max{$f$ is indefinitely differentiable in $\Omega$}.
    \end{prop}
    \Max{\begin{proof}
     The existence and uniqueness of a solution $f\in W^{1,\infty}_\mathrm{loc}(\Omega)$ is a combination of Proposition~\ref{p:existLyap} and Lemma~\ref{l:lyapFP}. From there in any smooth compact subdomain  $K\subset\subset\Omega$, we get from standard elliptic theory \cite{gilbarg_2001_elliptic} that $f\in\mathcal{C}^\infty(K)$, since the coefficents are smooth and the operator is uniformly elliptic.
    \end{proof}}

    \section{Noise reduction by binding : the case of fast $\mu$RNA}\label{sec:noise_fast}
    
    In this section we focus on the comparison between the explicit distributions \eqref{eq:rho0} and \eqref{eq:rhofast}. We are providing theoretical and numerical \Max{evidence}   that the coefficient of variation (which is a normalized standard deviation) of \eqref{eq:rhofast} is less than that of \eqref{eq:rho0}. This quantity called cell to cell variation in the biological literature \cite{osella_2011_role} characterizes the robustness of the gene expression level (the lower the better). We start by performing a rescaling in order to extract the dimensionless parameters which characterize the distributions.
    \subsection{Dimensional analysis}
    In order to identify the parameters of importance in the models, we rescale the variable\Max{s} $r$ \Max{and $\mu$} around characteristic value\Max{s} $\bar{r}$ \Max{and $\bar{\mu}$} chosen to be 
    \begin{equation}\label{eq:rcharac}
        \bar{r} = \frac{c_r}{k_r}\Max{\quad \text{and} \quad \bar{\mu}\ =\ \frac{c_\mu}{k_\mu}}\,.
    \end{equation}
    Th\Max{ese} choice\Max{s are}  natural in the sense that \Max{they} correspond to the steady state\Max{s} of the mRNA \Max{and $\mu$RNA}  dynamics without binding nor stochastic effects, that is \Max{respectively}  $\dd r_t = (c_r - k_r\ r_t)\dd t$ \Max{and $\dd \mu_t = (c_\mu - k_\mu\ \mu_t)\dd t$ }. \Max{When the noise term is added, it still corresponds to the expectation of the invariant distribution, that is the first moment of $\rho_0$ in the case of mRNA.} 
    \Max{We introduce $f^\text{ad}$ such that for all $(r,\mu)\in\Omega$ one has
    \[
    \frac{1}{\bar{r}\bar{\mu}}\,f^\text{ad}\left(\frac{r}{\bar{r}}, \frac{\mu}{\bar{\mu}}\right)\ =\ f(r,\mu)\,.
    \]
    After some computations one obtains that the Fokker-Planck equation \eqref{eq:FPstat}-\eqref{eq:defFP} can be rewritten in terms of $f^\text{ad}$ as
     \begin{multline}\label{eq:FPtrunc_adim}
            \ds\partial_r\left[\delta\,(1 - \gamma\,p\,r\,\mu - r)f^\text{ad} - \partial_r(r^2f^\text{ad})\right]\\+\ \partial_\mu\left[\delta\kappa(1 - \gamma\,r\,\mu - \mu)f^\text{ad} - \nu\,\partial_\mu(\mu^2f^\text{ad})\right]\ =\ 0\,,
    \end{multline}}    
    \Max{The marginal distributions} $\rho_0$ and $\rhofast$ are rescaled into dimensionless densities 
    \begin{align}
        \rho_0^{\delta}(r)\ &=\ g_{1+\delta, \delta}(r)\ =\ C_0^\text{ad}\,\frac{1}{r^{2+\delta}}\,\exp\left(-\frac{\delta}{r}\right)\label{eq:rho0adim}\\[.75em]
        \rhofast^{\delta,\gamma,p}(r)\ &=\ C_\text{fast}^\text{ad}\,\left(1+\frac{1}{\gamma r}\right)^{\gamma\,p\,\delta}\frac{1}{r^{2+\delta}}\,\exp\left(-\frac{\delta}{r}\right)\label{eq:rhofastadim}
    \end{align}
    where $C_0^\text{ad}$ and $C_\text{fast}^\text{ad}$ are normalizing constants depending on the parameters of the model and $\delta$, $p$ and $\gamma$ are dimensionless parameters. The first parameter
    \begin{equation}\label{eq:delta}
        \delta\ =\ \frac{k_r}{\sigma_r}\,,
    \end{equation}
    only depends on constants that are independent of the dynamics of $\mu$RNAs. The two other dimensionless parameters \Max{appearing in the marginal distribution of mRNA in the presence of fast $\mu$RNA} are
    \begin{equation}\label{eq:p}
        p\ =\ \frac{c_\mu}{c_r}\,,
    \end{equation}
    and
    \begin{equation}\label{eq:gamma_param}
        \gamma\ =\ \frac{c\,\bar{r}}{k_\mu}\ =\ \frac{c\,c_r}{k_\mu\,k_r}\,.
    \end{equation}
    Let us give some insight into the biological meaning of these parameters.
    The parameter $\gamma$ measures the relative importance of the two mechanisms of destruction of $\mu$RNAs, namely the binding with mRNAs
    versus the natural destruction/consumption. A large $\gamma$ means that the binding effect is strong and conversely. The parameter $p$ compares the production rate of $\mu$RNAs with that of mRNAs. Large values of $p$ mean that there are much more $\mu$RNAs than mRNAs produced per unit of time.
    
    \Max{Finally, in the Fokker-Planck model \eqref{eq:FPtrunc_adim}, there are also two other parameters which are 
        \begin{equation}\label{eq:kappa}
    \kappa\ =\ \frac{k_\mu}{k_r}\,,
    \end{equation}
    and 
    \begin{equation}\label{eq:nu}
    \nu\ =\ \frac{\sigma_\mu}{\sigma_r}\,.
    \end{equation}
    The parameter $\kappa$ compares  consumption of $\mu$RNA versus that of mRNA by mechanisms which are not the binding between the two RNAs. The parameter $\nu$ compares the amplitude of the noise in the dynamics of $\mu$RNA versus that of the mRNA.
    
    \begin{rema}
    Observe that the approximation of fast $\mu$RNA leading to the model discussed in Section~\ref{sec:fast} in its dimensionless form amounts to taking $\nu = \kappa = 1/\varepsilon$ and letting $\varepsilon$ tend to $0$.
    \end{rema}}

    \subsection{Cell to cell variation (CV)}
    For any suitably integrable non-negative function $\nu$, let us denote by 
    \[
    m_k(\nu)\ =\ \int y^k\,\nu(y)\dd y 
    \]
    its $k$-th moment. The coefficient of variation or cell to cell variation (CV) is defined by
    \begin{equation}\label{eq:def_CV}
        \mathrm{CV}(\nu)\ =\ \frac{\mathrm{Var}(\nu/m_0(\nu))^{1/2}}{\mathrm{Exp}(\nu/m_0(\nu))}\ =\ \left(\frac{m_2(\nu)m_0(\nu)}{m_1(\nu)^2}-1\right)^{1/2}
    \end{equation}
    where $\mathrm{Exp}(\cdot)$ and $\mathrm{Var}(\cdot)$ denote the expectation and variance. 
    \Max{Let us state a first lemma concerning some cases where the coefficient of variation can be computed exactly. } 
    \begin{lem}\label{lem:asymp}
        Consider the dimensionless distributions defined in \eqref{eq:rho0adim} and \eqref{eq:rhofastadim}. Then one has that
        \begin{equation}\label{eq:exp_var_cv}
            \mathrm{Exp}(\rho_0^{\delta})\ =\ 1\,,\quad \mathrm{Var}(\rho_0^{\delta})\ =\ \frac{1}{\delta-1}\,, \quad \mathrm{CV}(\rho_0^{\delta})\ =\ \frac{1}{\sqrt{\delta-1}}\,,
        \end{equation}
        where the variance and coefficient of variation are well-defined only for $\delta>1$. 
        Then for any $\delta > 1$, the following limits holds
        \[
        \begin{array}{rcl}
         \ds\lim_{\gamma\to0}\mathrm{CV}(\rhofast^{\delta,\gamma,p})& =& \mathrm{CV}(\rho_0^{\delta})\,,\quad \forall p>0\,,\\
         \ds\lim_{p\to0}\mathrm{CV}(\rhofast^{\delta,\gamma,p})& =& \mathrm{CV}(\rho_0^{\delta})\,,\quad \forall \gamma>0\,,\\
         \ds\lim_{\gamma\to\infty}\mathrm{CV}(\rhofast^{\delta,\gamma,p})& =& \mathrm{CV}(\rho_0^{\delta})\,,\quad \forall p\in[0,1)\,.
         \end{array}
        \]

    \end{lem}
    \begin{proof}
        The formulas for the moments follow from \eqref{eq:firstmoment} and  \eqref{eq:secondmoment}. Then observe that for all $r$, one has
        \[
        \lim_{\gamma\to0}\rhofast^{\delta,\gamma,p}(r)\ =\ \lim_{p\to0}\rhofast^{\delta,\gamma,p}(r)\ =\rho_0^{\delta}(r)
        \]
        and
        \[
        \lim_{\gamma\to+\infty}\rhofast^{\delta,\gamma,p}(r)\ =\ g_{1+\delta, (1-p)\delta}(r)
        \]
        and one can then take limits in integrals by dominated convergence.
    \end{proof}
    
    \Max{Let us give a biological interpretation of the previous lemma. When $\gamma=0$ or $p=0$, which respectively corresponds to the cases where there is no binding between mRNA and $\mu$RNA or there is no production of $\mu$RNA, the coefficient of variation is unchanged from the case of free mRNAs. The last limit states that if the $\mu$RNA production is weaker than the mRNA production, then in the regime where all $\mu$RNA is consumed by binding with mRNA, the coefficient of variation is also unchanged.}

    \Max{Outside of these asymptotic regimes, the theoretical result one would like to have is the following.
    \begin{conj}\label{conj}
     For any $\delta >1$, $\gamma,p>0$ and one has $\mathrm{CV}(\rhofast^{\delta,\gamma,p})\ \leq\ \mathrm{CV}(\rho_0^{\delta})$.
    \end{conj}
    At the moment, we are able to obtain the following uniform in $\gamma$ and $p$ bound
 \begin{equation}\label{eq:bestboundyet}
            \mathrm{CV}(\rhofast^{\delta,\gamma,p})\ \leq\ C_\delta\ \Max{:=}\ \left(\left(\frac{\delta}{\delta-1}\right)^2\left(1-\frac{1}{(\delta-1)^2}\right)^{{\delta-2}}-1\right)^{\frac12}\,,
        \end{equation}
        which holds for all $\delta>2$, $\gamma>0$ and $p\geq0$. The result is proved in  Proposition~\ref{prop:CV} in the Appendix. Observe that $C_\delta\geq \mathrm{CV}(\rho_0^{\delta})$ but asymptotically 
        \[
        C_\delta\sim_{\delta\to\infty}\mathrm{CV}(\rho_0^{\delta})\ =\ \frac{1}{\sqrt{\delta-1}}\,,
        \]
        so that $C_\delta$ is fairly close to  $\mathrm{CV}(\rho_0^{\delta})$ for large $\delta$.
        }
     In the next section we provide numerical \Max{evidence} that it should be possible to improve \Max{the right-hand side of} \eqref{eq:bestboundyet} \Max{and prove Conjecture~\ref{conj}}. \Max{Let us also mention that using integration by parts formulas it is possible to establish a recurrence relation between moments. From there one can infer the inequality of Conjecture~\ref{conj} for subsets of parameters $(\gamma, p)$. As the limitation to these subsets is purely technical and do not have any particular biological interpretation we do not report these results here.}

    \subsection{Exploration of the parameter space}\label{sec:exploration}
    
    Now, we explore the space of parameters $(\delta,\gamma,p)$ in order to compare the cell to cell variation in the case of fast $\mu$RNA and in the case of free mRNA.
    
    In order to evaluate numerically the cell to cell variation we need to compute $m_k(\rhofast^{\delta,\gamma,p})$, for $k=0,1,2$. Observe that after a change of variable these quantities can be rewritten (up to an explicit multiplicative constant depending on parameters)
    \[
    I_k = \int_0^\infty f_k(s)\,s^{\delta-2}e^{-s}\dd s\,,
    \]
    with $f_k(s) = s^{2-k}(1+s/(\gamma\delta))^{p\gamma\delta}$. For the numerical computation of these integrals, we use a Gauss-Laguerre quadrature
    \[
    I_k\ \approx\ \sum_{i=1}^N\omega_i^Nf_k(x_i^N)\,.
    \]
    which is natural and efficient as we are dealing with functions integrated against a gamma distribution. We refer to \cite{nist_2010} and references therein for the definition of the coefficients $\omega_i^N$ and quadrature points $x_i^N$. The truncation order $N$ is chosen such that the numerical error between the approximation at order $N$ and $N+1$ is inferior to the given precision $10^{-8}$ when $p\leq1$. For $p\geq1$, the function $f_k$ may take large values and it is harder to get the same numerical precision. In the numerical results below the mean error for the chosen sets of parameters with large values of $p$ is around $10^{-4}$ and the maximal error is $10^{-2}$. This is good enough to comment on qualitative behavior. \Max{The code used for these numerical simulations is publicly available on GitLab \cite{code}.}
    
    We plot the relative cell to cell variation $\mathrm{CV}(\rhofast^{\delta,\gamma,p}) / \mathrm{CV}(\rho_0^{\delta})$ with respect  to $\gamma$ and $p$ for two different values of $\delta$. The results are displayed on Figure \ref{fig:cvsurf}. 
    Then, on Figure \ref{fig:cvcurve}, we draw the explicit distributions $\rhofast^{\delta,\gamma,p}$ for various sets of parameters and compare it with $\rho_0^{\delta}$.
    
    \begin{figure}
        \begin{tabular}{cc}
            $\delta = 2$ & $\delta = 20$\\
            \includegraphics[width=.5\textwidth]{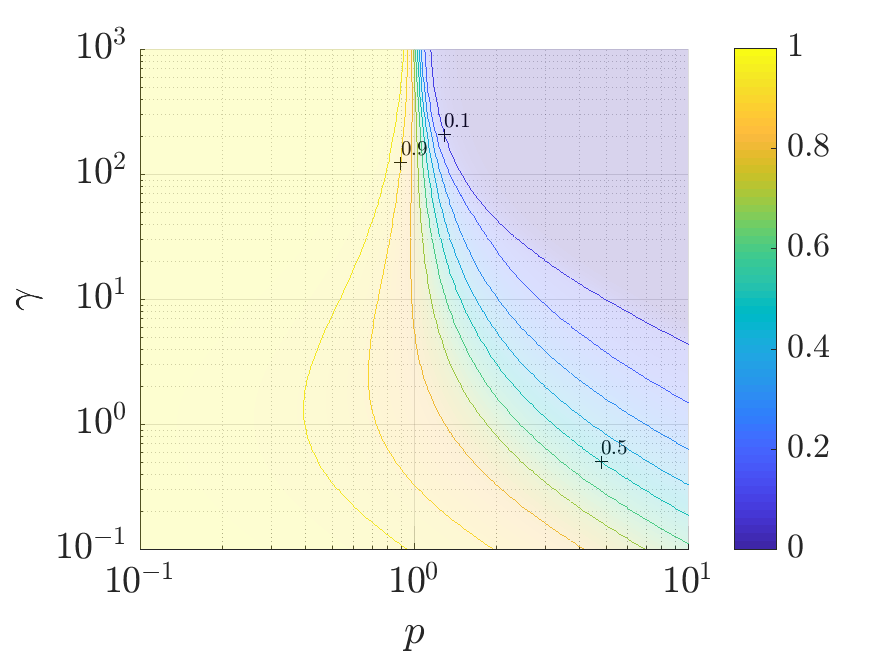} & \includegraphics[width=.5\textwidth]{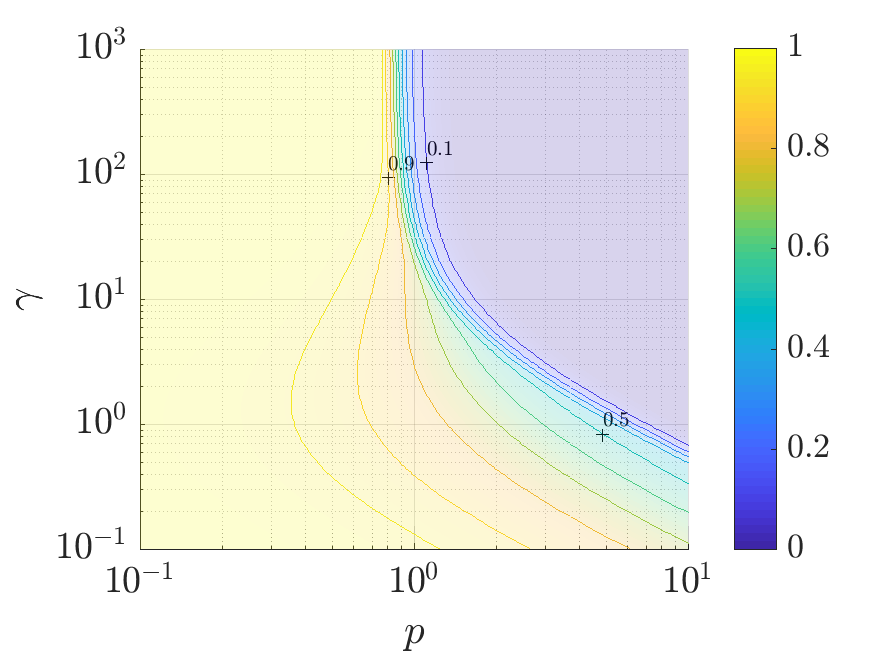}
        \end{tabular}
        \caption{\label{fig:cvsurf}\textbf{Exploration of the parameter space.} Relative cell to cell variation $\mathrm{CV}(\rhofast^{\delta,\gamma,p}) / \mathrm{CV}(\rho_0^{\delta})$ for various parameters $p$, $\gamma$ and $\delta$. On the horizontal axis, left means more production of mRNA and right means more production of $\mu$RNA; On the vertical axis, top means more destruction of mRNA by binding and bottom means more destruction/consumption of mRNA by other mechanisms}
    \end{figure}
    
        The numerical simulations of Figure~\ref{fig:cvsurf} suggest that the bound \eqref{eq:bestboundyet} is non-optimal \Max{and Conjecture~\ref{conj} should be satisfied. Observe also that the asymptotics of Lemma~\ref{lem:asymp} are illustrated.}
        
        From a modeling point of view, these simulations confirm that for any choice of parameter, the presence of (fast) $\mu$RNA makes the cell to cell variation decrease compared to the case without $\mu$RNA. Moreover, the qualitative behavior with respect to the parameters makes sense. Indeed we observe that whenever enough $\mu$RNA is produced ($p\geq1$), the increase of the binding phenomenon ($\gamma\to\infty$) makes the cell to cell variation decay drastically. 
    
    \begin{figure}
        \begin{tabular}{cc}
            \includegraphics[width=.5\textwidth]{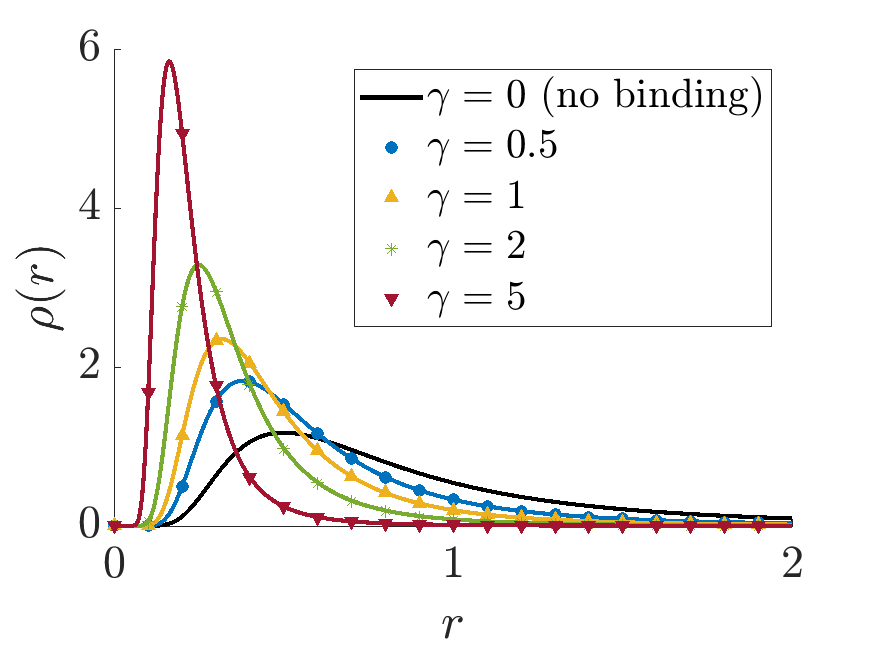} & \includegraphics[width=.5\textwidth]{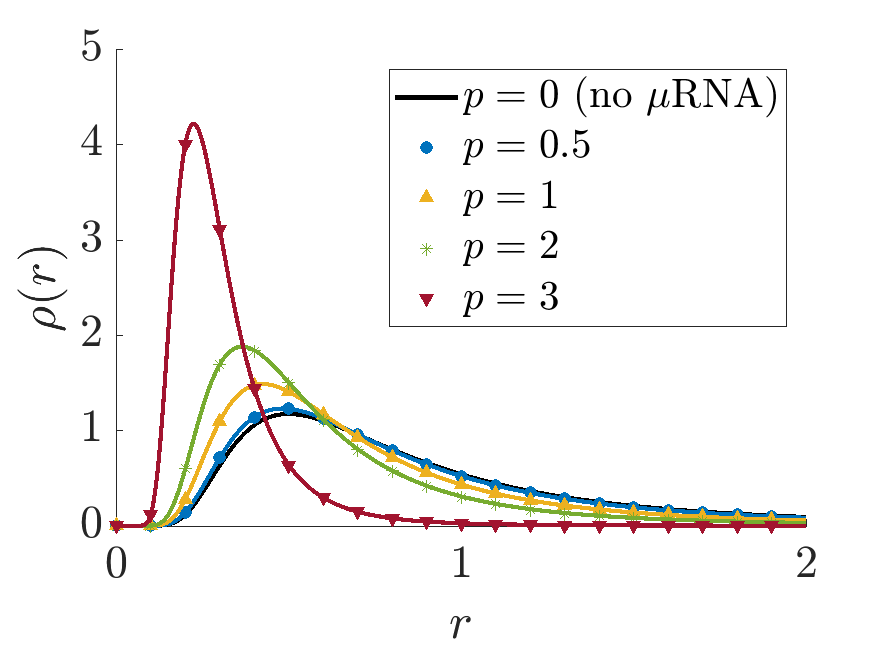}
        \end{tabular}
        \caption{\label{fig:cvcurve}Marginal distributions of mRNAs $\rhofast^{\delta,\gamma,p}$ for fast $\mu$RNAs compared to the free mRNA distribution $\rho_0^{\delta}$ (black solid curve) for different parameters $p$ and $\gamma$. Left: $\delta=2$, $p=1.5$ and $\gamma$ varies. Right: $\delta=2$, $\gamma=1$ and $p$ varies.}
    \end{figure}
    
    \section{Noise reduction by binding for the main Fokker-Planck model: numerical \Max{evidence}  }\label{sec:noise_main}
    In this section, we compute the gene expression level of the main model described by equation \eqref{eq:FPstat}. In this case, as there is no explicit formula for the solution, we will compute an approximation of it using a discretization of the Fokker-Planck equation. In order to compute the solution in practice, we restrict the domain to the bounded domain $\Omega_b = [r_{\min},r_{\max}]\times[\mu_{\min},\mu_{\max}]$. Because of the truncation, we add zero-flux boundary conditions in order to keep a conservative equation. It leads to the problem
    \begin{equation}
        \left\{
        \begin{array}{l}
            \ds\partial_r\left[(c_r - c\,r\,\mu - k_r\,r)f - \partial_r (\sigma_rr^2f)\right]\\[.5em]\quad\qquad+\ \partial_\mu\left[(c_\mu - c\,r\,\mu - k_\mu\,\mu)f - \partial_\mu( \sigma_\mu\mu^2f)\right]\ =\ 0\,,\quad \text{in }\Omega_b\\[.5em]
            \partial_r (\sigma_rr^2f) -(c_r - c\,r\,\mu - k_r\,r)f = 0\,,\quad\text{ if } r = r_{\min}\text{ or }r_{\max}\,, \\[.5em]
            \partial_\mu( \sigma_\mu\mu^2f) -(c_\mu - c\,r\,\mu - k_\mu\,\mu)f = 0\,,\quad\text{ if } \mu = \mu_{\min}\text{ or }\mu_{\max}\,,\\[.5em]
            \ds \int_{\Omega_b} f\dd r\dd\mu\ =\ 1\,. \
        \end{array}
        \right.
        \label{eq:FPtrunc}
    \end{equation}
    
    \subsection{\Max{Reformulation of the equation}}
    
    In order for the numerical scheme to be more robust with respect to the size of the parameters, \Max{we discretize the equation in dimensionless version \eqref{eq:FPtrunc_adim}}. It will also allow for comparisons with numerical experiments of the previous sections.
   
    As the coefficients in the advection and diffusion parts of \eqref{eq:FPtrunc_adim} grow rapidly in $r$, $\mu$ and degenerate when $r=0$ and $\mu=0$, \Max{the design of an efficient numerical solver for} 
    \eqref{eq:FPtrunc_adim} is not straightforward. Moreover a desirable feature of the scheme would be a preservation of the analytically known solution corresponding to $\gamma=0$. Because of these considerations we will discretize a reformulated version of the equation in which the underlying inverse gamma distributions explicitly appear. It will allow for a better numerical approximation when $r$ and $\mu$ are either close to $0$ or large.  The reformulation is the following
    \begin{equation}\label{eq:FPtrunc_adim_reform}
        \ds-\partial_r\left[r^2\,h^{(1)}(r,\mu)\partial_r\left(\frac{f^\text{ad}}{h^{(1)}(r,\mu)}\right)\right] \,-\,\partial_\mu\left[\nu\,\mu^2\,h^{(2)}(r,\mu)\partial_\mu\left(\frac{f^\text{ad}}{h^{(2)}(r,\mu)}\right)\right]\ =\ 0\,,
    \end{equation}
    with the associated no-flux boundary conditions and where the functions \Max{$h^{(1)}$} and \Max{$h^{(2)}$} are given by
    \begin{equation}\label{eq:h1}
        h^{(1)}(r,\mu)\ =\ r^{-(1+p\mu\gamma)\delta-2}\,\exp\left(-\frac{\delta}{r}\right)\,,
    \end{equation}
    and
    \begin{equation}\label{eq:h2}
        h^{(2)}(r,\mu)\ =\ \mu^{-(1+r\gamma)\delta\frac{\kappa}{\nu}-2}\,\exp\left(-\frac{\delta\kappa}{\nu r}\right)\,.
    \end{equation}

    \subsection{Presentation of the numerical scheme}
    
    We use a discretization based \Max{on the} reformulation \eqref{eq:FPtrunc_adim_reform}. It is inspired by \cite{bessemoulin2018hypocoercivity} and is fairly close to the so-called Chang-Cooper scheme \cite{chang1970practical}.

    We use a finite-volume scheme. The rectangle $\Omega_b$ is discretized with a structured regular mesh of size $\Delta r$ and  $\Delta \mu$ in each respective direction. The centers of the control volumes are the points $\left(r_i,\mu_j\right)$ with  $r_i = \Delta r/2 + i\Delta r$  and $\mu_j = \Delta \mu/2 + j\Delta \mu$ for \Max{$i\in\{0,\dots,N_r-1\}$} and \Max{$j\in\{0,\dots,N_\mu-1\}$}. We also introduce the intermediate points
    $r_{i+1/2}$ with $i\in\{-1,\dots,N_r-1\}$ and $\mu_{j+1/2}$ with $j\in\{-1,\dots,N_\mu-1\}$ defined with the same formula as before. The approximation of the solution on the cell $(i,j)$ is denoted \Max{by}
    \[
    f_{ij} \approx \frac{1}{\Delta r \Delta \mu} \int_{r_{i-1/2}}^{r_{i+1/2}}\int_{\mu_{j-1/2}}^{\mu_{j+1/2}} f^\text{ad}(r, \mu) \dd r \dd \mu.
    \]
    The scheme reads, for all \Max{$i\in\{0,\dots,N_r-1\}$} and \Max{$j\in\{0,\dots,N_\mu-1\}$},
    \begin{equation}
        \left\{
        \begin{array}{l}
            F_{i+1/2,j} - F_{i-1/2,j} + G_{i,j+1/2} - G_{i,j-1/2} = 0,\\[.5em]
            F_{N_r-1/2,j}\ =\ F_{-1/2,j}\ =\ G_{i,N_\mu-1/2}\ =\ G_{i,-1/2}\ =\ 0\\[.5em]
            \ds\sum_{i,j}f_{ij}\Delta r \Delta \mu \ =\ 1
        \end{array}
        \right.
        \label{eq:scheme}
    \end{equation}
    where the fluxes are given by a centered discretization of the reformulation \eqref{eq:FPtrunc_adim_reform}, namely
    \begin{equation}\label{eq:deffluxF}
        F_{i+1/2,j} = -\frac{\Delta \mu}{\Delta r}\,r_{i+1/2}^2\,\left(\frac{h^{(1)}(r_{i+1/2},\mu_j)}{h^{(1)}(r_{i+1},\mu_j)}f_{i+1,j}-\frac{h^{(1)}(r_{i+1/2},\mu_j)}{h^{(1)}(r_{i},\mu_j)}f_{ij}\right)\,,
    \end{equation}
    and
    \begin{equation}\label{eq:deffluxG}
        G_{i,j+1/2} =  -\nu\,\frac{\Delta r}{\Delta \mu}\,\mu_{j+1/2}^2\,\left(\frac{h^{(2)}(r_{i},\mu_{j+1/2})}{h^{(2)}(r_{i},\mu_{j+1})}f_{i,j+1}-\frac{h^{(2)}(r_{i},\mu_{j+1/2})}{h^{(2)}(r_{i},\mu_j)}f_{ij}\right)\,.
    \end{equation}
    One can show  that the scheme \eqref{eq:scheme} possesses a unique solution which is non-negative by following, for instance, the arguments of \cite[Proposition 3.1]{chainais_2011_finite}. Moreover, by construction, the scheme is exact in the case $\gamma=0$.
    
    \begin{rema}[Choice of $r_{\min}$, $r_{\max}$, $\mu_{\min}$, $\mu_{\max}$]
        Clearly $f$ decays faster at infinity than $\rho_0$ since the convection term coming \Max{from} the binding phenomenon brings mass closer to the origin. Therefore an appropriate choice for $r_{\max}$ and $\mu_{\max}$, coming from the decay of the involved inverse gamma distributions, should be (say) $r_{\max}^{-\delta}\leq 10^{-8}$ and $\mu_{\max}^{-\delta\kappa/\nu}\leq 10^{-8}$ so that the error coming from the tails of the distributions in the computation of moments is negligible. Similarly, near the origin the distributions decay very quickly to $0$ (as $\exp(-1/\cdot)$). Therefore $\mu_{\min}$, $r_{\min}$ can be taken not too small without influencing the precision in the computation of moments of the solution. In practice, we chose $\mu_{\min} = r_{\min} = 0.06$. Observe that even if nothing prevents the choice $\mu_{\min} = r_{\min} = 0$ on paper, one experiences in practice a bad conditioning of the matrix which has to be inverted for solving the scheme.
    \end{rema}

    \begin{rema}[Implementation]
     Observe that the matrix which has to be inverted in order to solve the scheme is not a square matrix because of the mass constraint (which is necessary to ensure uniqueness of the solution). In practice, in order to solve the corresponding linear system
     $MF = B$ where $F = (f_{ij})_{ij}$ and $B = (0,\dots,0, 1)\in\RR^{N_rN_\mu+1}$ and $M\in\RR^{(N_rN_\mu+1)\times N_rN_\mu}$ we use the pseudo-inverse yielding $F = (M^tM)^{-1}M^tB$. Finally the use of a sparse matrix routine greatly improves the computation time. Our implementation was made using Matlab. \Max{The code is publicly available on GitLab \cite{code}.}
    \end{rema}
  
    \subsection{Numerical results}
    
    \begin{figure}
        \begin{tabular}{ccc}
            \raisebox{-.5\height}{\includegraphics[width=.40\linewidth]{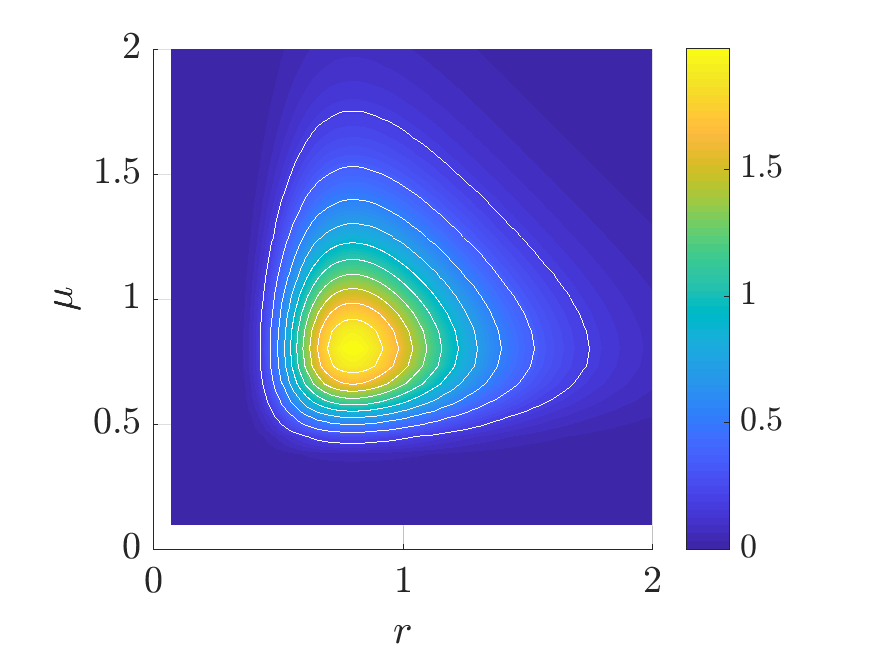}}&
            \begin{minipage}{.05\linewidth}${\scriptsize\gamma = 0}$\end{minipage}&\raisebox{-.5\height}{\includegraphics[width=.40\linewidth]{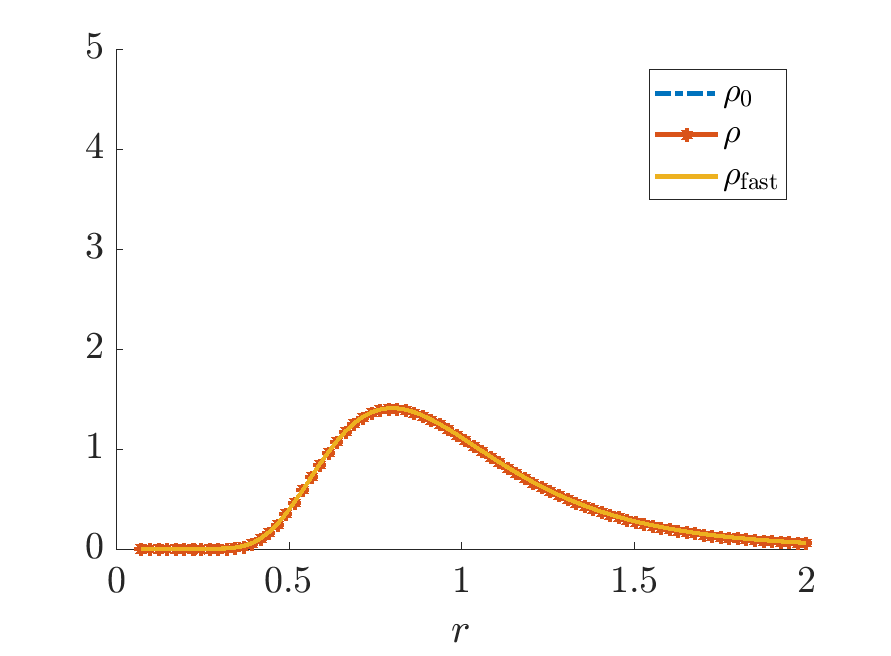}}\\\hdashline
            \raisebox{-.5\height}{\includegraphics[width=.40\linewidth]{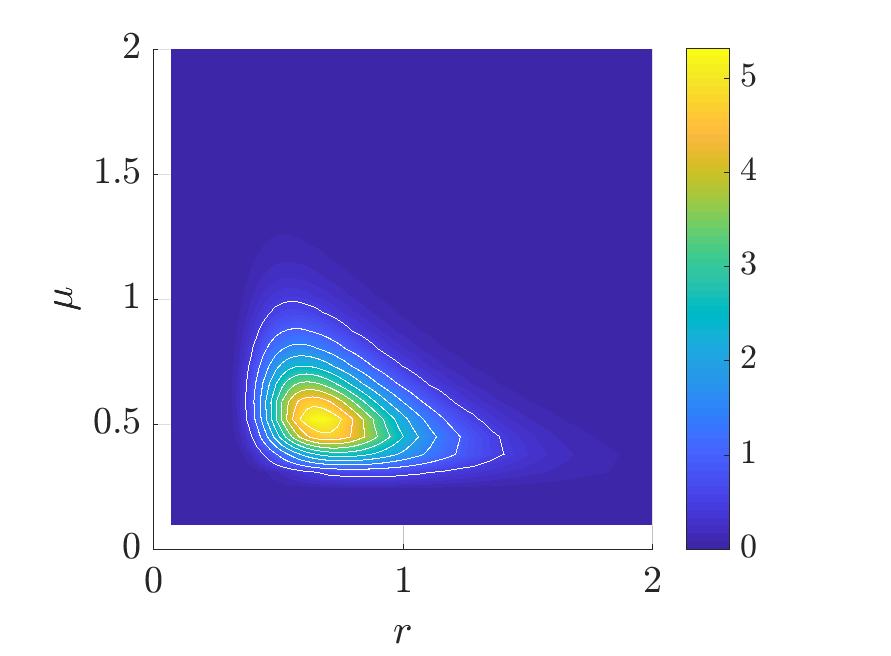}}&
            \begin{minipage}{.1\linewidth}${\scriptsize\gamma = 1}$\\${\scriptsize p = 0.5}$\end{minipage}& \raisebox{-.5\height}{\includegraphics[width=.40\linewidth]{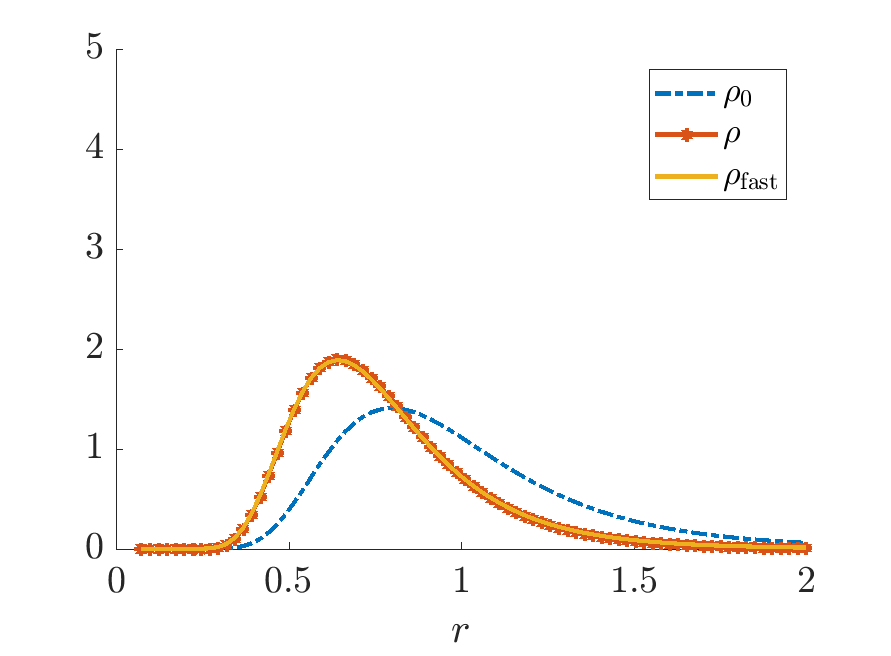}}\\\hdashline
            \raisebox{-.5\height}{\includegraphics[width=.40\linewidth]{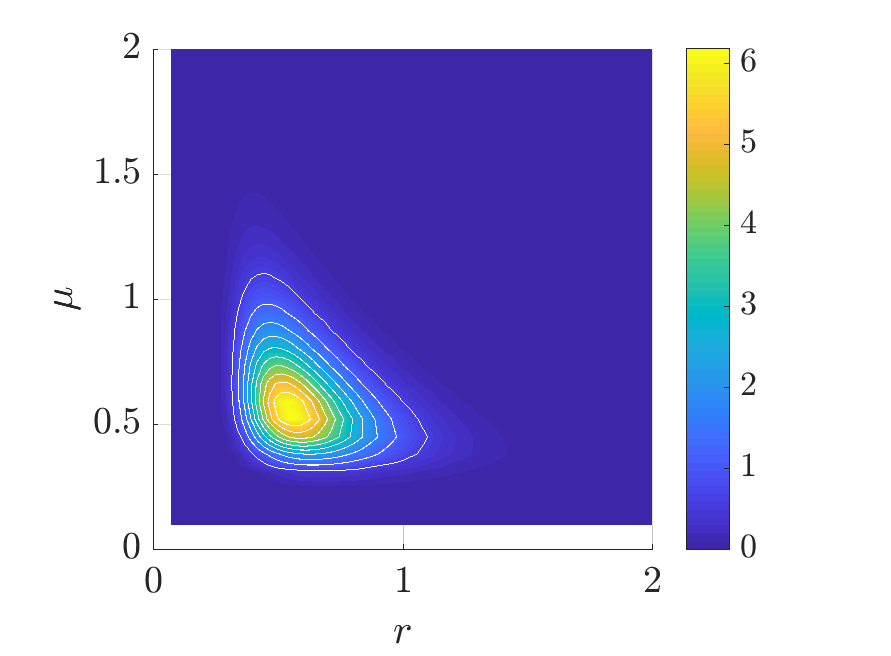}}&
            \begin{minipage}{.05\linewidth}${\scriptsize\gamma = 1}$\\${\scriptsize p = 1}$\end{minipage}&\raisebox{-.5\height}{\includegraphics[width=.40\linewidth]{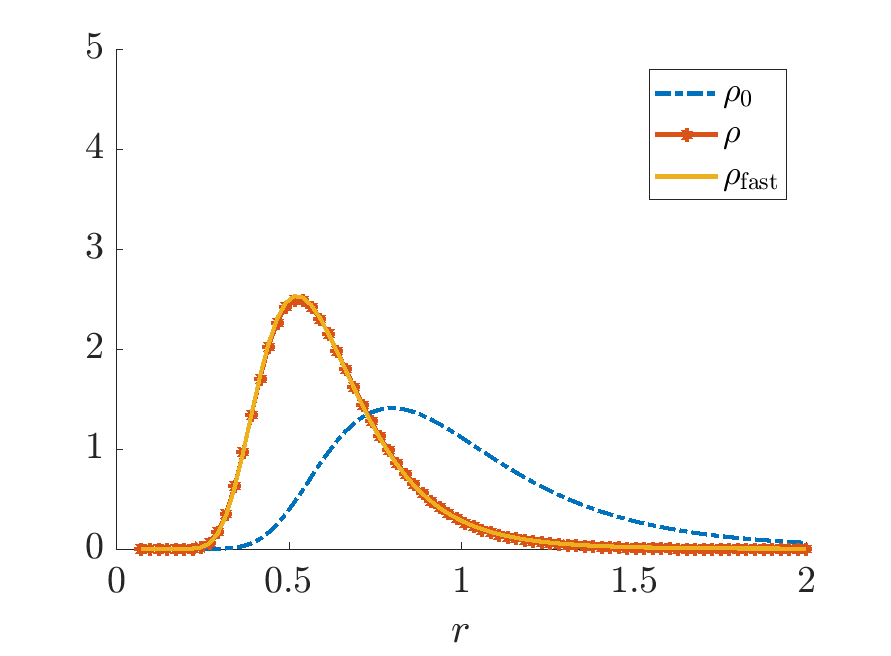}}\\\hdashline
            \raisebox{-.5\height}{\includegraphics[width=.40\linewidth]{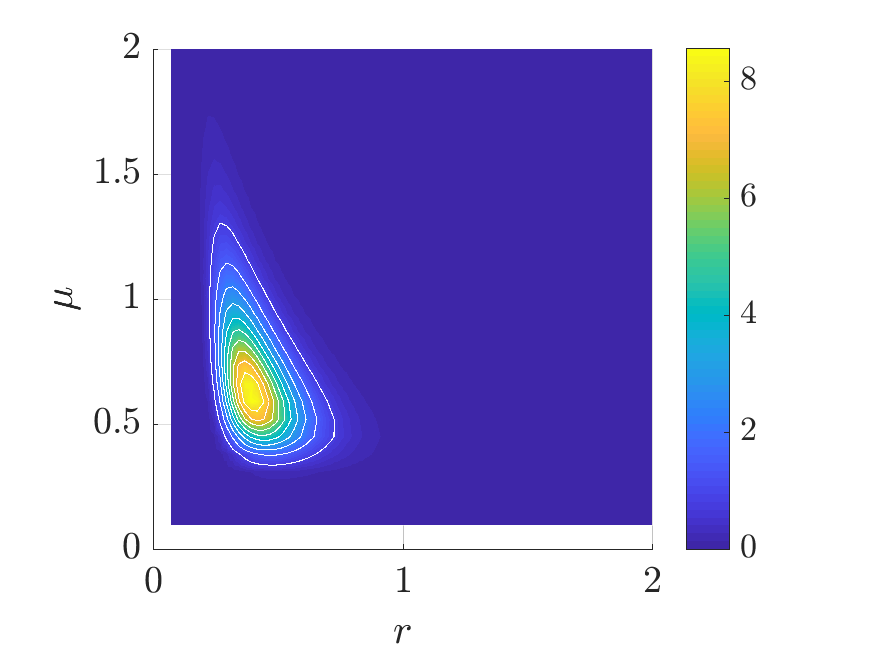}}&
            \begin{minipage}{.05\linewidth}${\scriptsize\gamma = 1}$\\${\scriptsize p = 2}$\end{minipage}&\raisebox{-.5\height}{\includegraphics[width=.40\linewidth]{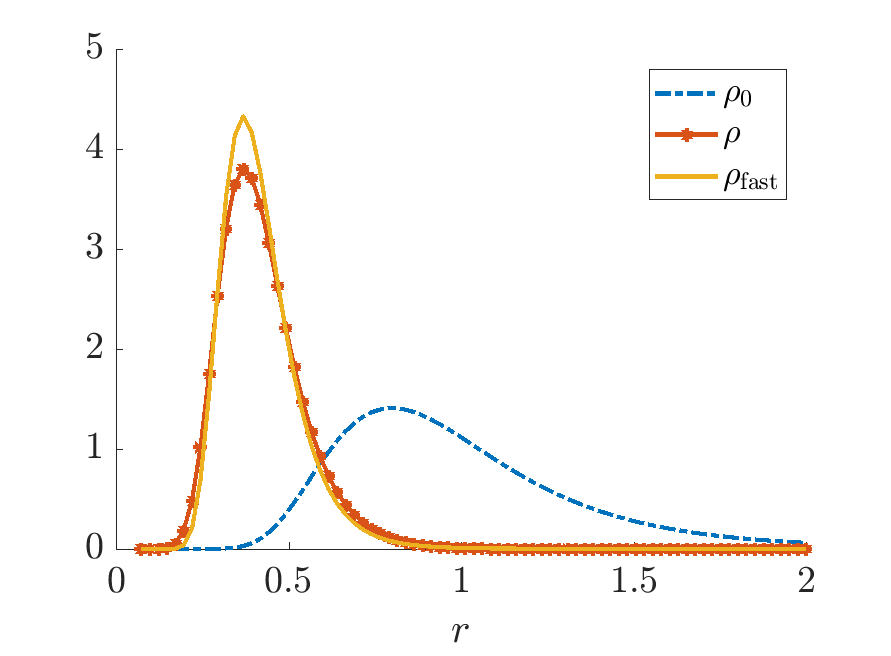}}\\\hdashline
            \raisebox{-.5\height}{\includegraphics[width=.40\linewidth]{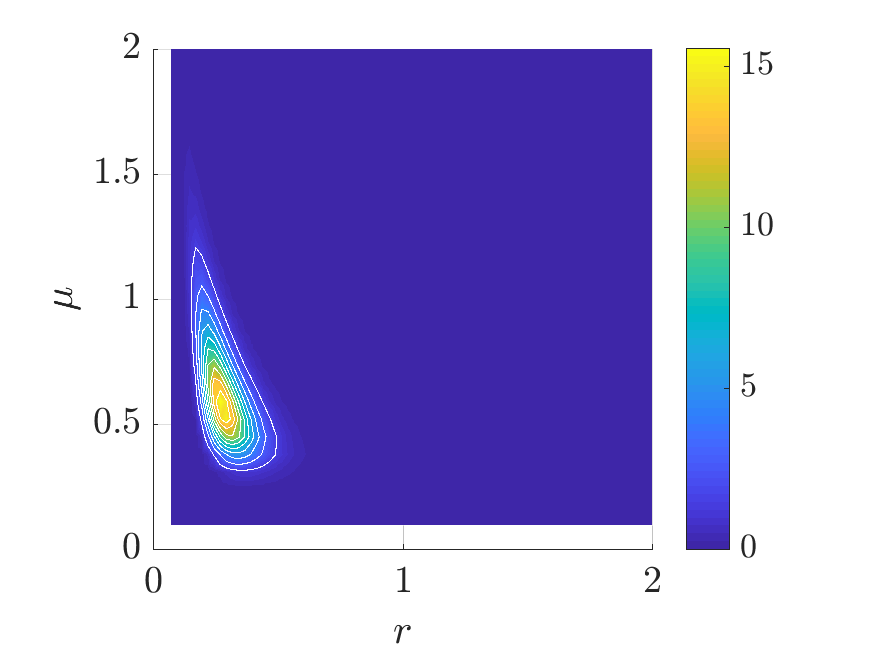}}&
            \begin{minipage}{.05\linewidth}${\scriptsize\gamma = 2}$\\${\scriptsize p = 2}$\end{minipage}&\raisebox{-.5\height}{\includegraphics[width=.40\linewidth]{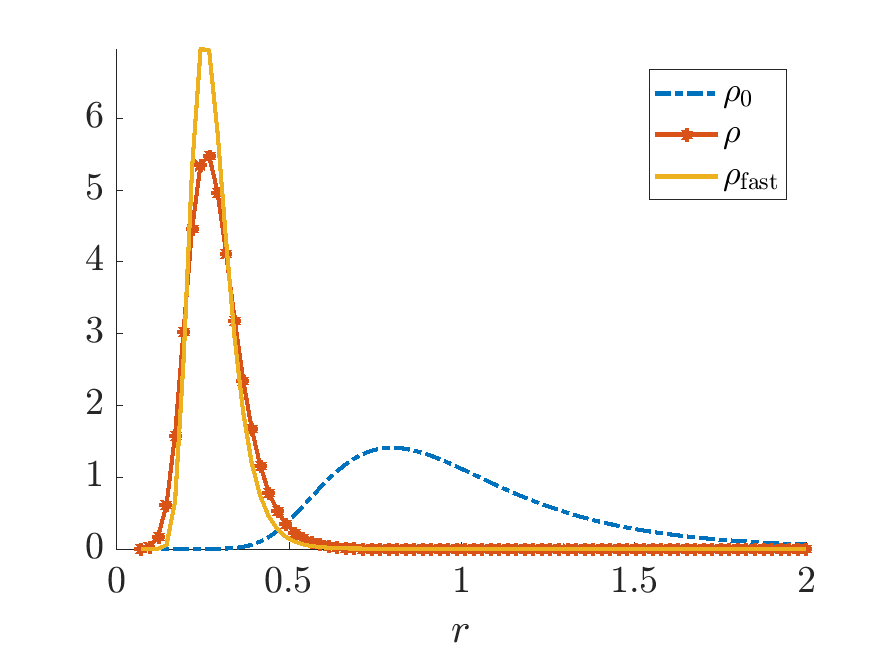}}
        \end{tabular}
        \caption{\label{fig:density}\textbf{Numerical results.} Numerical \Max{solution} 
        of the main Fokker-Planck model for various sets of parameters $(\gamma,p)$. \emph{Left:} Surface and contour plot of the distribution function $f(r,\mu)$. The truncation at $r=2$ and $\mu=2$ is only for visualization purposes. \emph{Right:} Corresponding marginal density $\rho$ compared with $\rhofast$ and $\rho_0$.}
    \end{figure}
       \begin{figure}
          \begin{tabular}{c}
          \begin{tabular}{cc}
              ${\scriptsize p=0.5}$&${\scriptsize p=2}$\\
           \includegraphics[width=.45\linewidth]{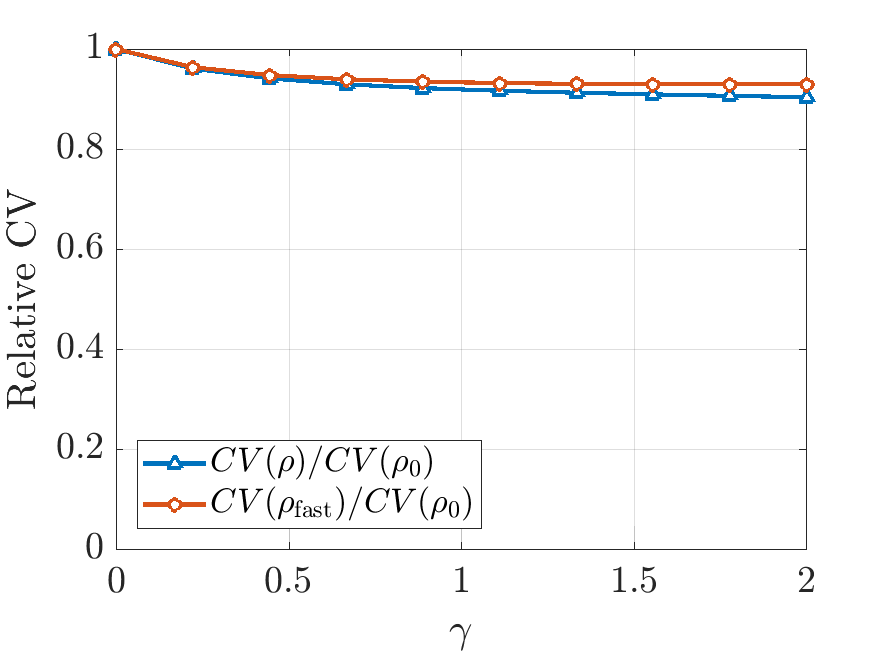}&\includegraphics[width=.45\linewidth]{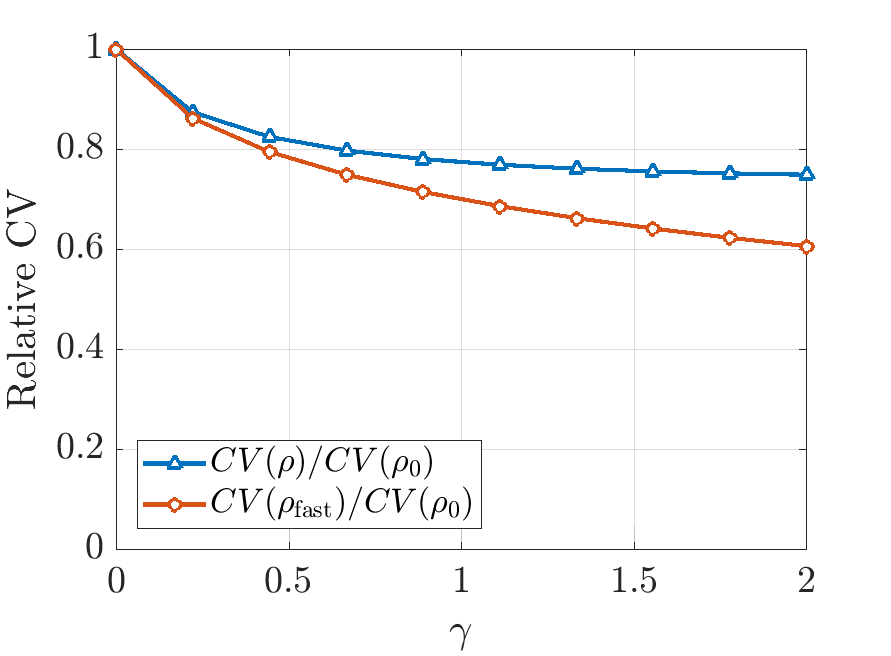}
          \end{tabular}\\
          \begin{tabular}{c}
           ${\scriptsize p=1}$\\
            \includegraphics[width=.45\linewidth]{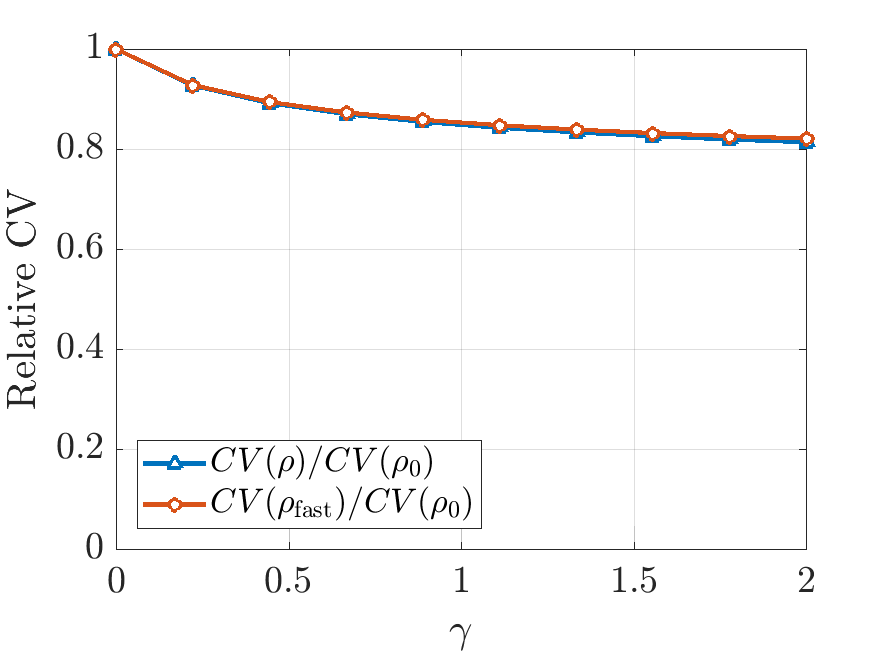}
          \end{tabular}
          \end{tabular}
          \caption{\label{fig:CV_mainFP}\textbf{Numerical results.} Relative coefficient of variation versus $\gamma$ for various values of $p$.}
      \end{figure}

    In our test cases we use the following parameters: $r_{\min} = 0.06$, $r_{\max} = 5$, $\mu_{\min} = 0.05$, $\mu_{\max} = 5$, $\delta = 8$, $N_r = 70$, $N_\mu = 200$, $\kappa = 1$, $\nu = 1$.
    
    On Figure~\ref{fig:density} we compare the distribution functions $f(r,\mu)$  obtained for various sets of parameters $(p,\gamma)$. We also draw the corresponding marginal $\rho(r)$ as well as $\rho_0$ and $\rhofast$. We observe that for small values of $p$, $\rhofast$ is a good approximation of $\rho$. For larger values it tends to amplify the phenomenon of variance reduction. 
    
    In order to confirm that the main Fokker-Planck model reduces the coefficient of variation as soon as $\gamma>0$ we draw on Figure~\ref{fig:CV_mainFP} the coefficient of variation for each distribution $\rho, \rhofast$ relatively to that of $\rho_0$ for several values of $p$. We observe that indeed, the coefficient of variation is reduced. As in the case of fast $\mu$RNA, the decay is more pronounced when the production of $\mu$RNA is higher than that of mRNA, namely when $p>1$. Interestingly enough, one also notices that the approximation $\rhofast$ increases the reduction of CV when $p>1$ and diminishes it when $p<1$. A transition at the special value $p=1$ was already observed on Figure~\ref{fig:cvsurf}.

    \section{Comments on the choice of noise}\label{sec:comments}
    
    In this section, we discuss the influence of the type of noise in the Fokker-Planck models. Let us go back to the system of stochastic differential equations considered at the beginning and generalize it as follows  
    \[
        \left\{
        \begin{array}{ccccc}
            \dd r_t&=&(c_r - c\, r_t\, \mu_t - k_r\, r_t)\, \dd t &+&\ \sqrt{2\,\sigma_r\,D(r_t)}\, \dd B_t^1\,,\\[1em]
            \dd\mu_t\, &=&\, (c_\mu\, - c\, r_t\, \mu_t - k_\mu\, \mu_t)\, \dd t &+& \sqrt{2\,\sigma_\mu\,D(\mu_t)}\, \dd B_t^2\,,
        \end{array}
        \right.
    \]
    with $D$ some given function. In the models of the previous sections we chose $D(x) = x^2$. On the one hand it is natural to impose that $D(x)$ vanishes when $x\to0$ in order to preserve the non-negativity of $r_t$ and $\mu_t$. On the other hand it is clear that the growth at infinity influences the tail of the equilibrium distribution which solves the corresponding Fokker-Planck equation. With a quadratic $D$ we obtained algebraically decaying distributions. 
    Nevertheless one may wonder if the decay of cell to cell variation due to $\mu$RNA would still be observed if $D$ is changed so that it involves distributions with faster decay at infinity. In order to answer this question, we choose a simple enough function $D$ so that we can still derive analytical formulas for distributions of mRNA without binding and mRNA in the presence of ``fast'' $\mu$RNA. Let us assume that
    \[
     D(r) = r\,.
    \]
    
    \subsection{Explicit formulas for distribution of mRNAs}
    
    In terms of modeling we may argue as in Section~\ref{sec:models} and Section~\ref{sec:solve} in order to introduce the stationary probability distribution of mRNA without binding $\tilde{\rho}_0$ which solves 
    \[
     \partial_r\left[\sigma_r\partial_r (r\trho_0) -(c_r - k_r\,r)\trho_0\right]\ =\ 0\,.
    \]
    It may still be solved analytically and one finds a gamma distribution
    \begin{equation}\label{eq:rho0_2}
    \trho_0(r)\ =\ \gamma_{\frac{c_r}{\sigma_r},\frac{k_r}{\sigma_r}}(r)\ =\ C_{\frac{c_r}{\sigma_r},\frac{k_r}{\sigma_r}}r^{\frac{c_r}{\sigma_r}-1}e^{-\frac{k_r}{\sigma_r}r}
    \end{equation}
    instead of an inverse gamma distribution in the quadratic case. The normalization constant is given in Section~\ref{sec:gamma}. 
    
    In the case of fast $\mu$RNA, we may once again follow the method of Section~\ref{sec:models} and Section~\ref{sec:solve} and introduce $\trhofast$ solving 
    \[
    \partial_r\left[\partial_r(\sigma_r\ r\ \trhofast) - (\ c_r - c\ r\ \tjfast(r) - k_r\ r)\trhofast  \right]\ =\ 0
    \]
    where the  conditional expectation of the number of $\mu$RNA within the population with $r$ mRNA is given by
    \[
    \tjfast(r)\ =\ \int_0^\infty \mu\,\gamma_{\frac{c_\mu}{\sigma_\mu},\frac{k_\mu + c\,r}{\sigma_\mu}}(\mu)\dd\mu\ =\ \frac{c_\mu}{k_\mu + c\,r}\,.
    \]
    A direct computation then yields 
    \begin{equation}\label{eq:rhofast_2}
        \trhofast(r)\ =\ C\,\left(1+\tfrac{c}{k_\mu}r\right)^{-\frac{c_\mu}{\sigma_r}}\,r^{\frac{c_r}{\sigma_r}-1}e^{-\frac{k_r}{\sigma_r}r}\,,
    \end{equation}
    with $C\equiv C(c_r,c_\mu, c, k_r, k\mu, \sigma_r, \sigma_\mu)$ a normalizing constant.
    \begin{rema}\label{rem:Ito}
     Observe that the  conditional expectation of the number of $\mu$RNA within the population with $r$ mRNA is unchanged, namely $\tjfast(r) = \jfast(r)$. \Max{More generally, the expectation of a univariate process $(X_t)_t$ satisfying an SDE with linear drift $\dD X_t = (a + bX_t)\dD t + \sqrt{2\sigma(X_t)}\dD B_t$ does not depend on the diffusion coefficient $\sigma$ as its density $g$ satisfy 
     $
      \partial_t g(t,x) + \partial_x((a+bx)g(t,x)) - \partial^2_{xx}(\sigma(x)g(t,x))=0\,,
     $
     so that multiplying by $x$ and integrating yields $\dD \mathrm{E}[X_t] = (a + b\mathrm{E}[X_t])\dD t$ on its expectation $\mathrm{E}[X_t]$. The argument also holds for multivariate processes.}

    \end{rema}
    
    \subsection{Dimensional analysis}
    Once again we seek the parameters of importance among the many parameters of the model by a dimensional analysis. The characteristic value of $r$ remains $\bar{r} = c_r/ k_r$ as it is the expectation of $\trho_0$. After rescaling we find the new distribution 
    \begin{equation}\label{eq:rho0adim_2}
        \trho_0^\eta(r)\ =\ \gamma_{\eta,\eta}(r)\ =\ C_{\eta,\eta}\,r^{\eta-1}e^{-\eta r}\,,
    \end{equation}
    and
    \begin{equation}\label{eq:rhofastadim_2}
        \trhofast^{\eta,\gamma,p}(r)\ =\ C_\text{fast}^\text{ad}\,\frac{r^{\eta-1}}{\left(1+\gamma r\right)^{p\eta}}\,e^{-\eta r}\,.
    \end{equation}
    where the parameters $p$ and $\gamma$ are given by \eqref{eq:gamma_param} and \eqref{eq:p}  respectively and still quantify the intensity of the binding and the respective production of $\mu$RNA versus mRNA. The new parameter $\eta$ is given by 
    \begin{equation}\label{eq:eta}
    \eta\ =\ \frac{c_r}{\sigma_r}\,.
    \end{equation}
    In the context of a dimensional analysis, let us mention that it would be inaccurate to compare $\eta$ and $\delta$ as the $\sigma_r$ (and $\sigma_\mu$) do not represent the same quantity depending on the choice of $D$. For $D(r) = r^2$ it has the same dimension as $k_r$ so $\delta = k_r/\sigma_r$ is the right dimensionless parameter. Here it has the same dimension as $c_r$, which justifies the introduction of $\eta$.

    \subsection{Numerical computation of the cell to cell variation}
    
    The expectation, variance and coefficient of variation of $\trho_0$ are explicitly given by
    \begin{equation}\label{eq:exp_var_cv_2}
        \mathrm{Exp}(\trho_0^{\eta})\ =\ 1\,,\quad \mathrm{Var}(\trho_0^{\eta})\ =\ \frac{1}{\eta}\,, \quad \mathrm{CV}(\trho_0^{\eta})\ =\ \frac{1}{\sqrt{\eta}}\,,
    \end{equation}
   
     \begin{figure}
        \begin{tabular}{cc}
            $\eta = 1$ & $\eta = 8$\\
            \includegraphics[width=.5\textwidth]{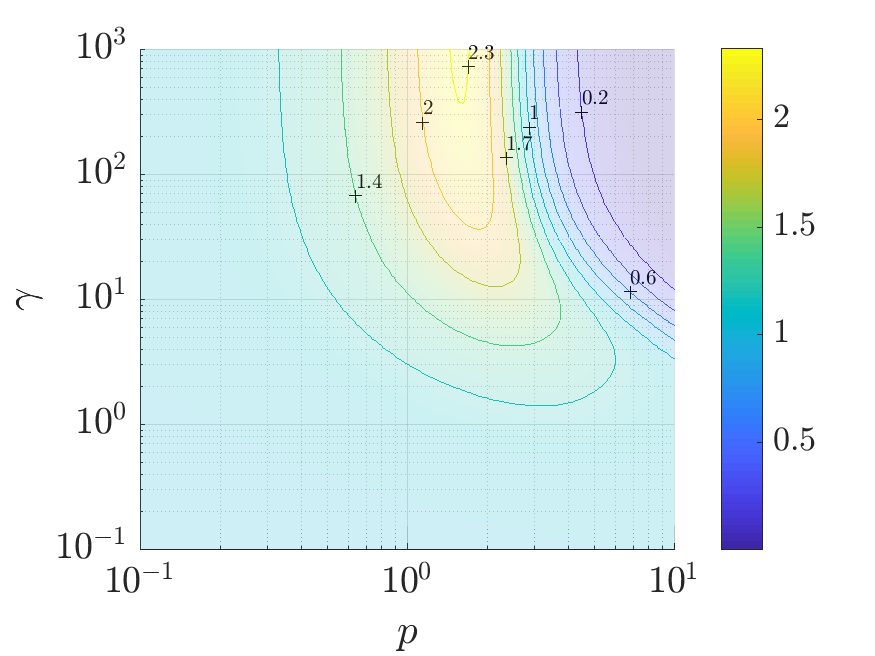} & \includegraphics[width=.5\textwidth]{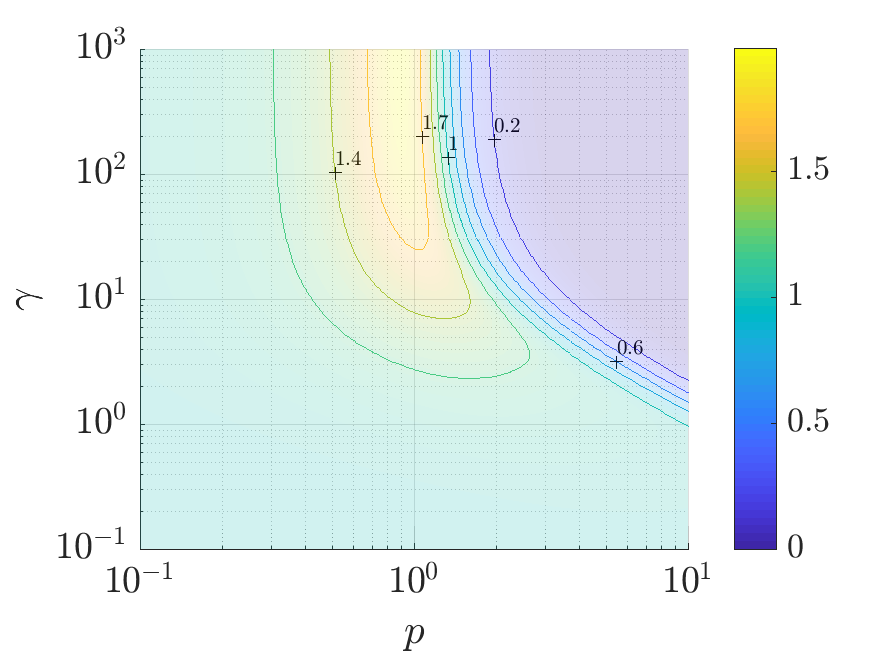}
        \end{tabular}
        \caption{\label{fig:cvsurf2}\textbf{Numerical computation of the cell to cell variation.} Relative cell to cell variation $\mathrm{CV}(\trhofast^{\eta,\gamma,p}) / \mathrm{CV}(\rho_0^{\eta})$ for various parameters $p$, $\gamma$ and $\eta$. On the horizontal axis, left means more production of mRNA and right means more production of $\mu$RNA; On the vertical axis, top means more destruction of mRNA by binding and bottom means more destruction/consumption of mRNA by other mechanisms.}
    \end{figure}
    
     As there is no explicit formula for the coefficient of variation of $\trhofast^{\eta,\gamma,p}$ we evaluate it numerically as in Section~\ref{sec:exploration}. The results are displayed on Figure~\ref{fig:cvsurf2}. We observe that unlike the case of a quadratic diffusion coefficient the relative cell to cell variation, \emph{i.e.} the cell to cell variation in the presence of $\mu$RNA relative to cell to cell variations of the free case, is not unconditionally less than $1$. For a large enough production of $\mu$RNA, it eventually decays when the binding effect is very strong. However for smaller production of $\mu$RNA or when the binding is weak, the effect is the opposite as the relative cell to cell variation is greater than $1$. This is not satisfactory from the modeling point of view.
     
     In conclusion the choice of noise is important in this model. An unconditional cell to cell variation decay in the presence of $\mu$RNA is observed for quadratic noise only. While other choices of noise may still lead to similar qualitative results, the choice $D(r) = r^2$ allowed us to derive explicit formulas for the approximate density $\rhofast$ which, as numerical simulations show, is fairly close to the marginal $\rho$ corresponding to the solution of the main Fokker-Planck model.

    \section{Concluding remarks and perspectives}
    
    In this paper, we introduced a new model describing the joint probability density of the number of mRNA and $\mu$RNA in a cell. It is based on a Fokker-Planck equation arising from a system of chemical kinetic equations for the number of two RNAs. The purpose of this simple model was to provide a mathematical { framework to investigate how robustness in gene expression in a cell is affected by the presence of a} regulatory feed-forward loop due to production of $\mu$RNAs which bind to and deactivate mRNAs.
    
    { Thanks to the combined use of analytical formulas and numerical simulations, we showed that robustness of gene expression is indeed affected by the presence of a feed-forward loop involving $\mu$RNA production. However, whether the effect is regulatory or de-regulatory strongly depends on the assumptions made on the type of noise affecting both mRNA and $\mu$RNAs production. In the case of geometric noise (the diffusivity being quadratic in the solution itself), the effect is to reduce the spread of the distribution as the reduction of the coefficient of variation shows. In the case of sub-geometric noise (the diffusivity being only linear in the solution itself), the effect increases the spread as shown by the increase of the coefficient of variation. We may attempt an explanation by comparing the  mRNA distribution in the absence of $\mu$RNA and in the limit of fast $\mu$RNA in the two cases. In the quadratic diffusivity case, both distributions are fat-tailed (i.e. they decay polynomially with the number of unbound mRNA molecules $r$, see \eqref{eq:rho0} and \eqref{eq:rhofast}) but the rate of decay at infinity is modified by the presence of $\mu$RNAs. On the other hand, in the linear diffusivity case, both decay exponentially fast (see \eqref{eq:rho0_2} and \eqref{eq:rhofast_2}) and the exponential rate of decay is the same with or without $\mu$RNAs. We propose that this might be the reason of the difference: in the quadratic diffusivity case, the change in polynomial decay allows to greatly reduce the standard deviation without affecting too much the mean, which results in a reduction of the coefficient of variation. In the linear diffusivity case, the exponential tail is not modified, which implies that the core of the distribution must be globally translated towards the origin, which affects the mean and the standard deviation in a similar way and does not systematically reduce the coefficient of variation. Which type of noise corresponds to the actual data is unknown at this stage. While quadratic diffusivity seems a fairly reasonable assumption (it is used in a number of contexts such as finance), it would require further experimental investigations to be fully justified in the present context. This discussion shows that the effect of $\mu$RNA on noise regulation of mRNA translation is subtle and not easily predictable.  }
		
		Along the way we provided theoretical tools for the analysis of the Fokker Planck equation at play and robust numerical methods for simulations. As the main biological hypothesis for the usefulness of $\mu$RNA in the regulation of gene expression is based on their ability to reduce external noise, we also discussed the particular choice of stochasticity in the model. 
    
    There are several perspectives to this work. A first one would be the calibration of the parameters of the model from real-world data. This would allow to quantify more precisely the amount of cell-to-cell variation reduction due to $\mu$RNA, thanks to the thorough numerical investigation done in this contribution of the effects of the parameters of the model. Besides, another perspective would be an improvement of Inequality~\eqref{eq:bestboundyet} \Max{to the Conjecture~\eqref{conj}}. This would bring a definitive theoretical answer to the hypothesis of increased gene expression level in the simplified model of ``fast'' $\mu$RNAs. One may also look into establishing a similar inequality for the general model. Finally, the gene regulatory network in a cell is considerably more complex than the simple, yet enlightening in our opinion, dynamics proposed in this paper. A natural improvement would be the consideration of more effects in the model, such as the production of the transcription factor, or the translation of mRNA into proteins, among many others.

    \appendix
    \Max{\section{Complementary results}}
    \subsection{Poincaré inequalities for gamma and inverse gamma distributions}\label{s:weightPoinc}
    
    In this section we give a elementary proof of the 1D version of the Brascamp and Lieb inequality (see \cite[Theorem 4.1]{brascamp_1976_extensions}, \cite{bobkov_2000_brunn}), which is an extension of the Gaussian Poincaré inequality in the case of log-concave measures. This allows us to derive a weighted Poincaré inequality for the gamma distribution and deduce, by a change of variable, a similar functional inequality for the inverse gamma distribution.
    
    \begin{prop}\label{p:poincare}
        Let $I\subset\RR$ be an open non-empty interval and $V:I\rightarrow\RR$ a function of class $\mathcal{C}^2$. Assume that
        \begin{itemize}
            \item[\emph{(i)}] $V$ is strictly convex;
            \item[\emph{(ii)}] $e^{-V}$ is a probability density on $I$;
            \item[\emph{(iii)}] $V$ tends to $+\infty$ at the extremities of $I$. 
        \end{itemize}
        Then, for any suitably integrable function $u$, one has
        \begin{equation}\label{eq:poincare_general}
            \int_I|u(x) - \lla\,u\,\rra_{e^{-V}}|^2\,e^{-V(x)}\,\dD x\ \leq\ \,\int_I|u'(x)|^2\,e^{-V(x)}\,(V''(x))^{-1}\,\dD x\,,
        \end{equation}
        where for a density $\nu$ the notation $\lla\,u\,\rra_{\nu}$ denotes $\int u\nu$.
    \end{prop}
    \begin{proof} 
        Without loss of generality, as one may replace $u$ with $u-\lla\,u\,\rra_{e^{-V}}$, we assume that $\lla\,u\,\rra_{e^{-V}} = 0$. We also assume that $u$ is of class $\mathcal{C}^1$ and compactly supported in $I$ and one can then extend \emph{a posteriori} the class of admissible function by a standard density argument. Then, using (ii) one has
        \[
        \begin{array}{rcl}
            \ds\int_I |u(x)|^2\,e^{-V(x)}\,\dD x &=& \ds \frac{1}{2}\,\iint_{I\times I}|u(x)-u(y)|^2\,e^{-(V(x)+V(y))}\,\dD x\,\dD y\\[1em]
            &=& \ds \frac{1}{2}\,\iint_{I\times I}\left|\int_x^yu'(z)\dD z\right|^2\,e^{-(V(x)+V(y))}\,\dD x\,\dD y\,.
        \end{array}
        \]
        Now using the Cauchy-Schwarz inequality and assumption (i) one has
        \begin{multline*}
            \int_I |u(x)|^2\,e^{-V(x)}\,\dD x\ \leq\\ \frac{1}{2}\,\iint_{I\times I}\left(\int_x^y|u'(z)|^2 (V''(z))^{-1}\,\dD z\right)\left(\int_x^yV''(z)\dD z\right)\,e^{-(V(x)+V(y))}\,\dD x\,\dD y\,.
        \end{multline*}
        Then take any point $x_0\in I$ and define
        \[
        U(x) = \int_{x_0}^x|u'(z)|^2\,(V''(z))^{-1}\,\dD z\,,
        \]
        so that the inequality rewrites
        \[
        \begin{array}{rcl}
            &&\ds\int_I |u(x)|^2\,e^{-V(x)}\,\dD x\\[1em] &\leq& \ds \frac{1}{2}\,\iint_{I\times I}\left(U(y)-U(x)\right)\left(V'(y)-V'(x)\right)\,e^{-(V(x)+V(y))}\,\dD x\,\dD y\,.
        \end{array}
        \]
        Now just expand the right-hand side and use Fubini's theorem on each term as well as assumptions (ii) and (iii) to obtain
        \[
        \int_I |u(x)|^2\,e^{-V(x)}\,\dD x\  \leq\  \int_I U(x)\,V'(x)\,e^{-V(x)}\,\dD x\,.
        \]
        One concludes by integrating the right-hand side by parts and observing that boundary terms vanish again by assumption (iii).
    \end{proof}
    
    \begin{rema} The proof is an adaptation of the original proof of the (flat) Poincaré-Wirtinger inequality by Poincaré.
        
        Observe that for $I=\RR$ and $V(x) = x^2/2$, one recovers the classical Gaussian Poincaré inequality.
    \end{rema}
    \begin{rema}
        The inequality is sharp. It is an equality for functions of the form $u(x) = aV'(x)+b$, with $a,b\in\RR$ if $V$ is such that $V'(x)e^{-V(x)}$ tends to $0$ at the boundaries, and only for constant functions otherwise (\emph{i.e.} $a=0$ and $b\in\RR$). 
    \end{rema}
    
    From the Brascamp-Lieb inequality, we now infer Poincaré inequalities for gamma  and inverse gamma distributions.
    
    \begin{prop}\label{cor:poincare_gamma}
        Let $\alpha>1$ and $\beta>0$. Then, for any functions $u,v$ such that the integrals make sense, one has 
        \begin{equation}
            \int_0^\infty|u(x) - \lla\,u\,\rra_ {\gamma_{\alpha,\beta}}|^2\,\gamma_{\alpha,\beta}(x)\,\dD x\ \leq\ \frac{1}{\alpha-1}\,\int_0^\infty|u'(x)|^2\,\gamma_{\alpha,\beta}(x)\,x^2\,\dD x\,,
            \label{eq:weightPoincare_G}
        \end{equation}
        and
        \begin{equation}
            \int_0^\infty|v(y) - \lla\,v\,\rra_{g_{\alpha,\beta}}|^2\,g_{\alpha,\beta}(y)\,\dD y\ \leq\ \frac{1}{\alpha-1}\,\int_0^\infty|v'(y)|^2\,g_{\alpha,\beta}(y)\,y^2\,\dD y\,,
            \label{eq:weightPoincare_IG}
        \end{equation}
        where for a probability density $\nu$ the notation $\lla\,u\,\rra_{\nu}$ denotes $\int u\nu$.
    \end{prop}
    \begin{proof}
        The first inequality is an application of \eqref{eq:poincare_general} with  $V(x) = \beta x - (\alpha-1)\ln(x) - \ln(C_{\alpha,\beta})$, where $C_{\alpha,\beta}=\beta^\alpha/\Gamma(\alpha)$. Then take $v(y) = u(1/y)$ and make the change of variable $y = 1/x$ in all the integrals of \eqref{eq:weightPoincare_G} to get the result.
    \end{proof}

    \begin{rema}\label{r:bakem}
        To the best of our knowledge the classical Bakry and Emery method does not seem to apply to show directly the functional inequalities of Proposition~\ref{cor:poincare_gamma}. Let us give some details. In order to show a Poincaré inequality of the type
        \[
        \int_I|u(x) - \lla\,u\,\rra_{e^{-V}}|^2\,e^{-V(x)}\,\dd x\ \leq\ \int_I|u'(x)|^2\,e^{-V(x)}\,D(x)\,\dD x\,,
        \]
        for $V$ as in Proposition~\ref{p:poincare}, it is sufficient that $D$ and $V$ satisfy the following curvature-dimension inequality
        \begin{multline}\label{eq:curv_dim}
            R(x)\ :=\ \frac{1}{4}(D'(x))^2-\frac{1}{2}D''(x)D(x)+D(x)^2V''(x)\\+\frac{1}{2}D'(x)D(x)V'(x)
            \ \geq\ \lambda_1\,D(x)\,,
        \end{multline}
        for some positive constant $\lambda_1>0$. We refer to \cite{bakry_1985_diffusions, bakry_1994_hypercontractivite} for the general form of the latter Bakry-Emery condition (for multidimensional anisotropic inhomogeneous diffusions) and to \cite{arnold_2001_convex} or \cite{arnold_2005_refined} for the simpler expression in the case of isotropic inhomogeneous diffusion, as discussed here. In the case of the inequalities \eqref{eq:weightPoincare_G} and \eqref{eq:weightPoincare_IG}, one has respectively $D(x) = x^2/(\alpha-1)$, $V(x) = \beta x - (\alpha-1)\ln(x) - \ln(C_{\alpha,\beta})$ and $D(y) = y^2/(\alpha-1)$, $V(y) = \beta / y + (\alpha+1)\ln(y) - \ln(C_{\alpha,\beta})$, which yields respectively $R(x) = \beta\,x^3/(\alpha-1)^2$ and $R(y) = \beta\,y/(\alpha-1)^2$. As claimed above, neither \eqref{eq:weightPoincare_G} nor \eqref{eq:weightPoincare_IG} satisfy the condition \eqref{eq:curv_dim}. One also observes that in both cases the curvature-dimension inequality fails because of a degeneracy at one end of the interval.

        Let us finally mention that there are in the literature other occurrences of Poincaré and more generally  convex Sobolev inequalities for gamma distributions \cite{bakry_1996_remarques, miclo_2003_inegalite, benaim_2008_exponential,arras_2017_stroll}. However, we found out that the diffusion coefficient is always taken of the form $D(x)=4x/\beta$. This weight, associated with the gamma invariant measure, corresponds to the Laguerre diffusion $L_{\alpha,\beta}f (x) = \beta x\,f'' (\beta x) - (\alpha - \beta x)f'(\beta x)$. This operator differs from the adjoint of the one appearing in our model \eqref{eq:FP2stat}. In this case, one can check that the curvature-dimension condition of Bakry and Emery is satisfied as soon as $\alpha\geq 1/2$. 
    \end{rema}
		
	\Max{\subsection{An upper bound for  the relative cell to cell variation}
	
	    \begin{prop}\label{prop:CV}
        One has the bound
        \[
            \mathrm{CV}(\rhofast^{\delta,\gamma,p})\ \leq\ C_\delta\ \Max{:=}\ \left(\left(\frac{\delta}{\delta-1}\right)^2\left(1-\frac{1}{(\delta-1)^2}\right)^{{\delta-2}}-1\right)^{\frac12}\,,
        \]
        which holds for all $\delta>2$, $\gamma>0$ and $p\geq0$.
    \end{prop}
    \begin{proof}
        The bound is a consequence of the Prékopa-Leindler inequality (see \cite{brascamp_1976_extensions} and references therein) which states that if $f,g,h:\RR^d\to[0,+\infty)$ are three functions satisfying for some $\lambda\in(0,1)$ and for all $x,y$,
        \begin{equation}\label{eq:cond_prekopa}
            h((1-\lambda)x+\lambda y)\,\geq\,f(x)^{1-\lambda}\,g(y)^{\lambda}\,,
        \end{equation}
        then 
        \begin{equation}\label{eq:ineg_prekopa}
            \|h\|_{L^1(\RR^d)}\ \geq\ \|f\|_{L^1(\RR^d)}^{1-\lambda}\,\|g\|_{L^1(\RR^d)}^{\lambda}\,.
        \end{equation}
        We use it with $\lambda = 1/2$,  $f(x) = (1+x)^{\gamma p\delta}x^{\delta-2}e^{-\delta \gamma x}$ if $x\geq0$ and $f(x) = 0$ if $x<0$, $g(x) =  x^2f(x)$ and $h(x) = (1+C_\delta^2)^{1/2}xf(x)$. The condition \eqref{eq:cond_prekopa} is then equivalent to 
        \[  
        (1+C_\delta^2)^{-1/2}\ \leq\ \left[\frac{\left(1+\frac{x+y}{2}\right)}{(1+x)(1+y)}\right]^{\frac{\gamma p\delta}{2}}\left(\frac yx\right)^{\frac12}\left(\frac{\left(\frac yx\right)^{\frac12}+\left(\frac xy\right)^{\frac12}}{2}\right)^{\delta-1}
        \]
        which is satisfied as the term between brackets is always greater than $1$ and the function $z\mapsto z[(z+z^{-1})/2]^{\delta-1}$, $z>0$ is bounded from below by $(1+C_\delta^2)^{-1/2}$, where $C_\delta$ is given in \eqref{eq:bestboundyet}. Then with the change of variable $x'= 1/(\gamma x)$ in the integrals of \eqref{eq:ineg_prekopa}, one recovers \eqref{eq:bestboundyet}.
    \end{proof}}

\medskip

\noindent
\textbf{Acknowledgements}	\\	
PD acknowledges support by the Engineering and Physical Sciences Research Council (EPSRC) under grants no. EP/M006883/1 and EP/N014529/1, by the Royal Society and the Wolfson Foundation through a Royal Society Wolfson Research Merit Award no. WM130048 and by the National Science  Foundation (NSF) under grant no. RNMS11-07444 (KI-Net). PD is on leave from CNRS, Institut de Math\'ematiques de Toulouse, France. MH acknowledges support by the Labex CEMPI (ANR-11-LABX-0007-01). SM acknowledges support by the 
CNRS–Royal Society exchange projects ``CODYN'' and ``Segregation models in social sciences'' and the Chaire Mod\'elisation Math\'ematique et Biodiversit\'e of V\'eolia Environment - \'Ecole Polytechnique - Museum National d’Histoire Naturelle - Fondation X. All authors would like to thank Prof. Matthias Merkenshlager from Imperial College Institute of Clinical Sciences for bringing this problem to their attention and stimulating discussions.

\medskip
\noindent
\textbf{Data statement}\\
No new data were collected in the course of this research.

    \bibliographystyle{plain}
    \bibliography{bibli}
    
\end{document}